\documentclass[final,leqno,onefignum,onetabnum]{siamltex1213}

\usepackage{amsmath}
\usepackage{amssymb}
\usepackage{algorithm}
\usepackage{algpseudocode}

\title{Input Selection for Performance and Controllability of Structured Linear Descriptor Systems\thanks{A preliminary version of this work appeared at the 51st IEEE Conference on Decision and Control (CDC), 2012.}}
\author{Andrew Clark, Basel Alomair, Linda Bushnell, and Radha Poovendran\thanks{A. Clark, L. Bushnell, and R. Poovendran are with the University of Washington, Seattle, Washington. Email: \{awclark, lb2,rp3\}@uw.edu.  B. Alomair is with the National Center for Cybersecurity Technology, King Abdulaziz City for Science and Technology, Riyadh, Saudi Arabia.   Email: alomair@kacst.edu.sa}}

\begin{document}
\maketitle
\slugger{mms}{xxxx}{xx}{x}{x--x}

\begin{abstract}
A common approach to controlling complex networks is to directly control a subset of input nodes, which then controls the remaining nodes via network interactions.  While techniques have been proposed for selecting input nodes based on either performance metrics or controllability, a unifying approach based on joint consideration of performance and controllability is an open problem.  In this paper, we develop a submodular optimization framework for selecting input nodes based on joint performance and controllability in structured linear descriptor systems.  We develop our framework for arbitrary linear descriptor systems.  In developing our framework, we first prove that selecting a minimum-size set of input nodes for controllability  is a matroid intersection problem that can be solved in polynomial-time in the network size.  We then prove that input selection to maximize a  performance metric with controllability as a constraint is equivalent to maximizing a monotone submodular function with two matroid basis constraints, and derive efficient approximation algorithms with provable optimality bounds for input selection.  Finally, we present a  graph controllability index metric, which characterizes the largest controllable subgraph of a given complex network, and prove its submodular structure, leading to input selection algorithms that trade-off performance and controllability.   We provide improved optimality guarantees for known systems such as strongly connected networks, consensus networks, networks of double integrators, and networks where all system parameters (e.g., link weights) are chosen independently and at random.

\end{abstract}



\pagestyle{myheadings}
\thispagestyle{plain}
\markboth{Input Selection for Performance and Controllability}{A. Clark, B. Alomair, L. Bushnell, and R. Poovendran}

\section{Introduction}
\label{sec:intro}
Complex networks consist of distributed nodes with locally coupled dynamics in domains including intelligent transportation systems \cite{wang1996coordination}, social networks \cite{ghaderi2012opinion}, energy systems \cite{rohden2012self}, and biological networks \cite{mangan2003structure}.  In many of these applications, the complex network must be controlled to reach a desired state, for example, steering a group of unmanned vehicles to a certain formation \cite{lawton2003decentralized}.	A common approach to controlling complex networks is to directly control one or more states of a subset of nodes, denoted as input nodes, while relying on the local coupling to drive the remaining nodes to the desired state \cite{hong2006tracking,ji2006leader}.

An important design parameter when controlling complex networks is the set of nodes that act as inputs.
 The problem of selecting a minimum-size set of input nodes to control a complex network has achieved significant  attention \cite{liu2011controllability,pasqualetti2013controllability,ruths2014control}, with  recent work focusing on selecting input nodes to satisfy  controllability, defined as the ability to drive the network from any initial state to any desired state in finite time using the input nodes.
 Since the seminal work of \cite{liu2011controllability}, a variety of discrete optimization methods have been proposed for selecting input nodes to achieve controllability \cite{olshevsky2014minimum,pasqualetti2013controllability,ruths2014control}.  

 In addition to controllability, networked systems are expected to satisfy performance criteria, including convergence rate of the dynamics \cite{pasqualetti2008steering} and robustness to noise and disturbances \cite{bamieh2011coherence,patterson2010leader}.  Approaches for selecting input nodes based on these criteria, including submodular optimization \cite{clark2014convergence,clark2014leader}, convex relaxation \cite{lin2013leader}, and combinatorial \cite{fitch2013information} algorithms, have been proposed in the literature, but are largely disjoint from the algorithms for selecting input nodes to satisfy controllability.	At present, a unifying and computationally efficient approach for selecting input nodes based on both controllability and performance is not available in the  literature.

	Current approaches to input selection assume either that the system matrices are fully known, or that only the structure of the system matrices is known.  In the latter case, the system matrix consists of zero entries and free parameters, where the free parameters can take any arbitrary value.
	Many practical systems, however, lie between these two extremes, with either a mixture of known and unknown matrix entries, or structural relationships between the unknown entries (i.e., all of the entries of a row sum to zero, as in the case of linear consensus algorithms \cite{jadbabaie2003coordination}).
 	In \cite{chapman2013strong}, it was observed that a set of selected input nodes may not satisfy controllability if these structural properties are not taken into account. 
 	While  input selection methods that incorporate system structure have been proposed for specific applications such as consensus \cite{goldin2013weight}, a computationally tractable general framework that guarantees controllability of structured systems remains an open problem.  

 \subsection{Our Contributions}
 \label{subsec:contributions}
In this paper, we develop a submodular optimization framework for input selection based on joint consideration of performance and controllability in structured linear descriptor systems. The submodular structure implies that a variety of input selection problems can be solved up to a $(1-1/e)$ optimality bound, using algorithms that depend on the constraints of each problem. We first  show that selecting the minimum-size set of input nodes to satisfy  controllability of structured linear descriptor systems can be mapped to a maximum-cardinality matroid intersection problem, leading to the first polynomial-time algorithm for ensuring controllability of such systems.  We then investigate selecting a set of up to $k$ input nodes to maximize a performance metric while satisfying a controllability constraint, and prove that this problem is equivalent to submodular maximization with two matroid basis constraints.  We develop a randomized algorithm for solving a relaxed, continuous version of the problem with a $(1-1/e)$ optimality bound, which can be rounded to a feasible input set that satisfies controllability.  As a third problem, we relax the requirement that the system is controllable and select input nodes based on a trade-off between performance and controllability.  In this case, we prove that the problem has the structure of submodular maximization subject to a cardinality constraint, leading to $(1-1/e)$ optimality bound.  

We next study input selection when the complex network is strongly connected (i.e., there exists a directed path between any two nodes).  We prove that, for almost all systems with a given structure, the controllability of the networked system can be represented as a single matroid constraint.  Based on this result, we derive a linear-time algorithm for selecting input nodes for controllability in structured linear descriptor systems, and prove that this algorithm is guaranteed to select the minimum-size input set.    We further show that the problem of selecting a set of up to $k$ input nodes to optimize performance subject to a controllability constraint can be solved with optimality bound of $(1-1/e)$ when the network is strongly connected.

  We investigate three special cases of our framework, arising from  classes of structured systems that have been studied in the existing literature, namely, linear consensus \cite{jadbabaie2003coordination}, second-order integrator dynamics \cite{ren2008double}, and systems in which all nonzero parameters can take arbitrary values \cite{liu2011controllability}.  We show that our general approach achieves at least the same optimality guarantees compared to the current state of the art for each individual problem.  Our results are illustrated via numerical study in the special case of a consensus network.


\subsection{Related Work}
\label{subsec:related}
Structural controllability of linear systems with given inputs was first studied by Lin in \cite{lin1974structural}, in the case where all nonzero matrix entries are independent free parameters.  Controllability of systems with additional relationships between the matrix entries was considered subsequently \cite{corfmat1976structurally,murota1987refined,reinschke1997digraph}.  In \cite{reinschke1997digraph}, graph-based conditions for structural controllability of descriptor systems were introduced.  The work of \cite{murota1987refined} provided a matroid-based framework for structural controllability of mixed-matrix descriptor systems, containing both fixed and free entries, as well as polynomial-time algorithms for verifying controllability of such systems.  For a detailed survey of controllability results in linear descriptor systems, see \cite{dion2003generic}.  In these works, conditions and algorithms for verifying structural controllability with a given input set are provided, but the problem of selecting the input nodes is not considered.

Selecting input nodes to satisfy controllability has been extensively studied in recent years.  Necessary and sufficient conditions for a set of input nodes to guarantee controllability in leader-follower consensus dynamics were presented in \cite{tanner2004controllability}.  Graph-based necessary conditions were derived in \cite{rahmani2009controllability}.  These works considered controllability from a given set of input nodes, but did not introduce efficient algorithms for selecting the input nodes.  In  \cite{liu2011controllability}, a polynomial-time graph matching algorithm was introduced for selecting a minimum-size set of input nodes to satisfy structural controllability.  The problem of selecting input nodes for controllability was further considered in \cite{clark2012leader,pasqualetti2013controllability,ruths2014control} for the case where all matrix entries are either zero or are free, independent parameters.  For the case where all entries of the system matrices are fully known, input selection algorithms and optimality bounds were derived in \cite{olshevsky2014minimum}.  In the present paper, we consider a broader class of system matrices that contain both free parameters and fixed entries, taking the existing works as special cases.  


Input selection based on performance criteria has also received research interest.  In \cite{clark2014convergence,clark2014leader}, submodular optimization approaches for selecting input nodes for robustness to noise and smooth convergence to a desired state were  developed.  A convex relaxation approach for minimizing errors due to link and state noise was proposed in \cite{fardad2011noisefree,lin2011noisecorrupted,lin2013leader}.  Combinatorial algorithms for input selection to minimize the $H_{2}$-norm, using information centrality, were introduced in \cite{fitch2013information}. 
\subsection{Organization}
\label{subsec:organization}
The paper is organized as follows.  In Section \ref{sec:preliminaries}, we present our system model, definitions and sufficient conditions for controllability, background on submodularity and matroid theory, and examples of performance metrics that can be incorporated into our framework.  In Section \ref{sec:general}, we develop our submodular optimization framework for input selection in structured systems.  Section \ref{sec:connected} discusses input selection in strongly connected network, and shows how connectivity improves the optimality bounds of our approach.  Section \ref{sec:cases} presents three special cases of our framework found in the existing literature.  Section \ref{sec:numerical} contains a numerical study.  Section \ref{sec:conclusion} concludes the paper. 
\section{Model and Preliminaries}
\label{sec:preliminaries}
In this section, we describe the system model and definitions of controllability considered.  Background on matroid theory and submodularity is given.  We then present sufficient conditions for controllability from previous work based on matroids and an auxiliary graph construction.  Finally, we give examples of performance metrics that can be incorporated into our framework.    

\subsection{System Model}
\label{subsec:sys_model}
We consider a linear, time-invariant networked system with a total of $n$ states, where $\mathbf{x}(t) \in \mathbb{R}^{n}$ is the vector of states at time $t$. In the absence of any control inputs, the system dynamics are described by
\begin{equation}
\label{eq:dynamics_no_inputs}
F\dot{\mathbf{x}}(t) = A \mathbf{x}(t).
\end{equation}
Eq. (\ref{eq:dynamics_no_inputs}) defines a \emph{linear descriptor system}~\cite{dion2003generic}.
The matrices  $F$ and $A$ can be further decomposed as $F = Q_{F} + T_{F}$ and $A = Q_{A} + T_{A}$.  The values of $Q_{F}$ and $Q_{A}$ are known, fixed parameters.  The matrices $T_{F}$ and $T_{A}$ are \emph{structure matrices}, in which any nonzero entry can take any arbitrary real value.  We  assume that the system defined by (\ref{eq:dynamics_no_inputs}) satisfies solvability, defined as follows.
\begin{definition}
\label{def:solvable}
An LTI system of the form (\ref{eq:dynamics_no_inputs}) is \emph{solvable} if for any initial state $\mathbf{x}(0) \in \mathbb{R}^{n}$, there exists a unique trajectory $\{\mathbf{x}(t) : t > 0\}$ that satisfies (\ref{eq:dynamics_no_inputs}).
\end{definition}

We now describe the effect of control inputs on the system (\ref{eq:dynamics_no_inputs}).  A subset $S$ of states act as control inputs, i.e., for each state $i \in S$, there exists a control input $u_{i}(t)$ such that $x_{i}(t) = u_{i}(t)$ for all $t$.  The states in $S$ correspond to the states of nodes that are controlled directly by an external entity.  In the following, without loss of generality, we assume that the state indices are ordered so that states $\mathbf{x}_{R} = (x_{1},x_{2}, \ldots, x_{n-|S|})$ do not act as control inputs, and states $\mathbf{x}_{S} = \{x_{n-|S|+1},\ldots,x_{n}\}$ act as control inputs.  The system dynamics are then given by
\begin{equation}
\label{eq:dynamics_input}
\left(
\begin{array}{c}
\hat{F} \\
\hline
0
\end{array}
\right)
\left(
\begin{array}{c}
\dot{\mathbf{x}}_{R}(t) \\
\dot{\mathbf{x}}_{S}(t)
\end{array}
\right) =
\left(
\begin{array}{c|c}
A_{RR} & A_{RS} \\
\hline
0_{(n-|S|) \times |S|} & -I_{|S| \times |S|}
\end{array}
\right)\left(
\begin{array}{c}
\mathbf{x}_{R}(t) \\
\mathbf{x}_{S}(t)
\end{array}
\right) + \left(
\begin{array}{c}
0_{(n-|S|) \times |S|} \\
I_{|S| \times |S|}
\end{array}
\right)
\mathbf{u}(t)
\end{equation}
In (\ref{eq:dynamics_input}), $\hat{F}$ is the $(n-|S|) \times n$ matrix consisting of the first $(n-|S|)$ rows of $F$.  The matrices $A_{RR}$ and $A_{RS}$ consist of the first $(n-|S|)$ and last $|S|$ columns of the first $(n-|S|)$ rows of $A$, respectively.  The vector $\mathbf{u}(t)$ is the control input signal.  The last $|S|$ rows of the equation enforce the condition that $x_{i}(t) = u_{i}(t)$ for all $i \in S$.

As a notation, for a square matrix $X \in \mathbb{R}^{n \times n}$, $N(X)$ is a graph with $n$ vertices, where there exists an edge $(i,j)$ if $X_{ji}$ is nonzero.

\subsection{Matroids and Submodularity}
\label{subsec:matroid}
In what follows, we define the concepts of matroids and submodular functions, which will be used in our optimization framework.  All definitions and lemmas in this subsection can be found in \cite{oxley1992matroid}. 
\begin{definition}
\label{def:matroid}
 A matroid $\mathcal{M} = (V,\mathcal{I})$ is defined by a finite set $V$ (denoted the ground set) and a collection $\mathcal{I}$ of subsets of $V$ such that (a) $\emptyset \in \mathcal{I}$, (b) if $X \subseteq Y$ and $Y \in \mathcal{I}$, then $X \in \mathcal{I}$, and (c) if $X, Y \in \mathcal{I}$ and $|X| < |Y|$, then there exists $v \in Y \setminus X$ such that $(X \cup \{v\}) \in \mathcal{I}$.  The collection $\mathcal{I}$ is denoted as the collection of independent sets of $\mathcal{M}$.
 \end{definition}

A maximal independent set is  a \emph{basis}; we let $\mathcal{B}(\mathcal{M})$ denote the set of bases of a matroid $\mathcal{M}$.
 The \emph{rank function} $\rho$ of a matroid is a function $\rho: 2^{V} \rightarrow \mathbb{Z}_{\geq 0}$, given by $\rho(X) = \max{\{|Y|: Y \subseteq X, Y \in \mathcal{I}\}}$.  The \emph{rank} of a matroid is equal to $\rho(V)$.    Matroids can be characterized by their rank functions, as demonstrated in the following lemma.

 \begin{lemma}
 \label{lemma:matroid_rank}
 Let $\rho: 2^{V} \rightarrow \mathbb{Z}_{\geq 0}$.  Suppose that (i) $\rho(\emptyset) = 0$, (ii) For any $X \subseteq V$ and $v \in V \setminus X$, $\rho(X) \leq \rho(X \cup \{v\}) \leq \rho(X) + 1$, and (iii) For any $X \subseteq V$ and $v,w \notin X$, if $\rho(X) = \rho(X \cup \{v\})$ and $\rho(X) = \rho(X \cup \{w\})$, then $\rho(X) = \rho(X \cup \{x,w\})$.  Then $\rho$ is the rank function of a matroid, which is denoted as the matroid induced by $\rho$.
 \end{lemma}

 A simple example of a matroid is the uniform matroid $U_{k}$, defined by $X \in U_{k}$ if $|X| \leq k$ for some $k$.  In a linear matroid, the set $V$ is equal to a collection of vectors in $\mathbb{R}^{m}$, and a set of vectors is independent if the vectors are linearly independent.  In this case, the rank function is equal to the column rank of the matrix defined by the vectors.
 The following is a method of constructing matroids that will be used in this paper.
 \begin{definition}
 \label{def:matroid_union}
 Let $\mathcal{M}_{1} = (V_{1}, \mathcal{I}_{1})$ and $\mathcal{M}_{2} = (V_{2}, \mathcal{I}_{2})$ be matroids.  Then the \emph{matroid union} $\mathcal{M} = \mathcal{M}_{1} \vee \mathcal{M}_{2}$ is defined by $V = V_{1} \cup V_{2}$ and $X \in \mathcal{I}$ if $X = X_{1} \cup X_{2}$ with $X_{1} \in \mathcal{I}_{1}$, $X_{2} \in \mathcal{I}_{2}$.
 \end{definition}

It can be shown  that the matroid union $\mathcal{M}$ is a matroid, with rank function $\rho(X) = \min{\{\rho_{1}(Y) + \rho_{2}(Y) + |X \setminus Y| : Y \subseteq X\}}$.  A second matroid construction is the \emph{dual matroid}, described as follows.
\begin{definition}
\label{def:matroid_dual}
Let $\mathcal{M} = (V,\mathcal{I})$ be a matroid.  The dual of $\mathcal{M}$, denoted $\mathcal{M}^{\ast}$, has ground set $V$ and set of independent sets $\mathcal{I}^{\ast}$ given by $$\mathcal{I}^{\ast} = \{X^{\prime} \subseteq X : V-X \emph{ is a basis of } \mathcal{M}\}.$$
\end{definition}
If $\rho$ is the rank function of $\mathcal{M}$, then the rank function $\rho^{\ast}$ of $\mathcal{M}^{\ast}$ is given by $$\rho^{\ast}(X) = \rho(V - X) + |X| - \rho(V).$$
We now define the concept of a \emph{submodular function}.

\begin{definition}
\label{def:submod}
Let $V$ be a finite set.  A function $f: 2^{V} \rightarrow \mathbb{R}_{\geq 0}$ is \emph{submodular} if for any sets $A$ and $B$ with $A \subseteq B$ and any $v \in V \setminus B$, $$f(A \cup \{v\}) - f(A) \geq f(B \cup \{v\}) - f(B).$$
\end{definition}

The following lemma gives a method for constructing submodular functions.

\begin{lemma}
\label{lemma:submod_comp}
If $f: 2^{V} \rightarrow \mathbb{R}$ is submodular as a function of $S$, then $f(V \setminus S)$ is submodular as a function of $S$.
\end{lemma}

For a submodular function $f: 2^{V} \rightarrow \mathbb{R}$, with $|V| = n$, the \emph{multilinear relaxation} $F: \mathbb{R}^{n} \rightarrow \mathbb{R}_{+}$ is defined as $$F(x) = \sum_{S \subseteq V}{\left[f(S) \left(\prod_{i \in S}{x_{i}}\right)\left(\prod_{i \notin S}{(1-x_{i})}\right)\right]}.$$

\subsection{Definitions of Controllability}
\label{subsec:controll_def}
We now present the definition of controllability considered in this work, which can be found in more detail in \cite{reinschke1997digraph}. For general systems of the form
\begin{equation}
\label{eq:LDS}
F\dot{\mathbf{x}}(t) = A\mathbf{x}(t) + B\mathbf{u}(t)
\end{equation}
we first have the following lemma.
\begin{lemma}[\cite{yip1981solvability}]
\label{lemma:solvability}
If the system (\ref{eq:LDS}) is solvable, then we can write $\mathbf{x} = [\mathbf{x}_{1} \ \mathbf{x}_{2}]^{T}$ and have matrices $F_{1}$, $F_{2}$, $B_{1}$, and $B_{2}$ such that the dynamics (\ref{eq:LDS}) are equivalent to
\begin{eqnarray*}
 \dot{\mathbf{x}}_{1}(t) &=& F_{1}\mathbf{x}_{1}(t) + B_{1}\mathbf{u}(t) \\
 F_{2}\dot{\mathbf{x}}_{2}(t) &=& \mathbf{x}_{2}(t) + B_{2}\mathbf{u}(t)
 \end{eqnarray*}
 \end{lemma}

 We now define the concepts of admissibility and reachability.
 \begin{definition}
 \label{def:admissible_reachable}
 An initial state $\mathbf{x}(0)$ is \emph{admissible} if $\mathbf{x}_{2}(0) = -\sum_{i=0}^{m-1}{F_{2}^{i}u^{(i)}(0)}$, where $m$ is the degree of nilpotency of $F_{2}$ and $u^{(i)}$ is the $i$-th derivative of the input $u$.  A state $\mathbf{x}^{\ast}$ is \emph{reachable} if there exists $t > 0$, an admissible initial state $\mathbf{x}_{0}$, and an input signal $\{\mathbf{u}(t^{\prime}) : t^{\prime} \in [0,t]\}$ such that $\mathbf{x}(t) = \mathbf{x}^{\ast}$ when $\mathbf{x}(0) = \mathbf{x}_{0}$.
 \end{definition}

 We let $R$ denote the set of reachable states.  We are now ready to define the concept of controllability.

 \begin{definition}
 \label{def:controllability}
 The system (\ref{eq:dynamics_input}) is \emph{controllable} if, for any admissible initial state $\mathbf{x}_{0}$ and any reachable final state $\mathbf{x}^{\ast}$ and time $t > 0$, there exists a control signal $\{\mathbf{u}(t^{\prime}) : t^{\prime} \in [0,t]\}$ such that $\mathbf{x}(0) = \mathbf{x}_{0}$ and $\mathbf{x}(t) = \mathbf{x}^{\ast}$.
 \end{definition}

 Finally, structural controllability is defined as follows.

 \begin{definition}
 \label{def:struct_controllability}
 The system (\ref{eq:LDS}) is \emph{structurally controllable} if there exist values for the free parameter matrices $T_{F}$, $T_{A}$, and $T_{B}$ such that (\ref{eq:LDS}) is controllable.  
 \end{definition}

 Structural controllability holds if the system (\ref{eq:dynamics_input}) is controllable for almost any choice of the free parameters.  Note that, if $Q_{A} = Q_{F} = 0$, then Definition \ref{def:struct_controllability} reduces to that of \cite{liu2011controllability}. A matrix pencil interpretation is given by the following theorem.

    \begin{theorem}[\cite{murota1987refined}]
   \label{theorem:matrix_pencil_controllability}
   The system (\ref{eq:LDS}) is structurally controllable if and only if there exist free parameter matrices $T_{F}$, $T_{A}$, and $T_{B}$ such that the following conditions hold:
   \begin{eqnarray}
   \label{eq:controll1}
   \mbox{$(F,A,B)$ is solvable} \\
   \label{eq:controll2}
   \mbox{rank}(A|B) = n \\
   \label{eq:controll3}
   \mbox{rank}((A-zF)|B) = n \quad \mbox{for all $z \in \mathbb{C}$}
   \end{eqnarray}
   \end{theorem}

We have assumed that (\ref{eq:controll1}) holds, since the system is solvable.  It remains to find equivalent or sufficient conditions for (\ref{eq:controll2}) and (\ref{eq:controll3}).  The following lemma gives a matroid-based sufficient condition for (\ref{eq:controll2}).

 \begin{lemma}[\cite{murota1987refined}]
 \label{lemma:condition_s_zero}
 Let $\mathcal{M}(I|Q_{A}|Q_{B})$ denote the linear matroid defined by the fixed parameter matrices $Q_{A}$ and $Q_{B}$, and let $\mathcal{M}(I|T_{A}|T_{B})$ denote the matroid defined by the free parameter matrices $T_{A}$ and $T_{B}$.  Then $\mbox{rank}(A|B) = n$ iff $\mbox{rank}(\mathcal{M}(I|Q_{A}|Q_{B}) \vee \mathcal{M}(I|T_{A}|T_{B})) = 2n$.
\end{lemma}

In the following section, we give a graph construction, introduced in \cite{murota1987refined}, that will be used to derive sufficient conditions for (\ref{eq:controll3}).

\subsection{Auxiliary Graph Construction}
\label{subsec:aux_graph}
We consider the following auxiliary graph constructed from (\ref{eq:LDS}).  Define a matrix
\begin{equation}
\label{eq:Omega_matrix}
\Omega =
\begin{array}{c}
\mathbf{w}: \\
\mathbf{x}: \\
\mathbf{u}:
\end{array}
\left(
\begin{array}{cc}
Q_{A}-Q_{F} & Q_{B} \\
I & 0 \\
0 & I
\end{array}
\right)
\end{equation}
Here, $\mathbf{w}$, $\mathbf{x}$, and $\mathbf{u}$ denote indices, so that the first $n$ rows are indexed $w_{1},\ldots,w_{n}$, the second $n$ rows are indexed $x_{1},\ldots,x_{n}$, and the third $k$ rows are indexed $u_{1},\ldots,u_{k}$.

As an intermediate step in the construction, define a bipartite graph $H$ with vertex set $$V_{H} = \{w_{1}^{T},\ldots,w_{n}^{T}\} \cup \{x_{1}^{Q}, \ldots, x_{n}^{Q}\} \cup \{w_{1}^{Q},\ldots,w_{n}^{Q}\}$$ and edge set $$E_{H} = \{(w_{i}^{T}, w_{i}^{Q}) : i=1,\ldots,n\} \cup \{(w_{i}^{T}, x_{j}^{Q}) : (i,j) \in N(T_{A}) \cup N(T_{F})\}.$$  The following lemma gives properties of matchings in this bipartite graph.

\begin{lemma}[\cite{murota1987refined}]
\label{lemma:bipartite_property}
If the system (\ref{eq:LDS}) is solvable, then there exists a perfect matching $m$ on the graph $H$ (i.e., a matching in which all nodes in $\{w_{1}^{T},\ldots,w_{n}^{T}\}$ are matched) such that the rows indexed in the set $$J = \{w_{i} : m(w_{i}^{T}) = w_{i}^{Q}\} \cup \{x_{i} : m(w_{j}^{T}) = x_{i}^{Q} \mbox{ for some $w_{j}$}\}$$ are linearly independent in $\Omega$.
\end{lemma}

Let $J$ be a set satisfying the conditions of Lemma \ref{lemma:bipartite_property} with matching $m$, and let $\Omega_{J}$ be a submatrix of $\Omega$ obtained from the rows in $J$. We have that $\Omega_{J}$ is an $n \times (n + k)$ matrix with full rank, and hence we can find rows $J_{1}$ such that $\Omega_{J \cup J_{1}}$ is a full-rank $(n+k) \times (n+k)$ matrix.  Let $\tilde{\Omega} = \Omega \Omega_{J \cup J_{1}}^{-1}$.  We index the columns of $\tilde{\Omega}$ in the set $J \cup J_{1}$.

We now define the auxiliary graph $\hat{G} = (\hat{V}, \hat{E})$ using the matrix $\tilde{\Omega}$.  The vertex set is equal to
\begin{multline*}
\hat{V} = \{w_{i}^{T} : i=1,\ldots,n\} \cup \{w_{i}^{Q}: i=1,\ldots,n\} \cup \{x_{i}^{T} : i=1,\ldots,n\} \cup \{x_{i}^{Q} : i=1,\ldots,n\}\\
 \cup \{u_{i}^{T} : i=1,\ldots,k\} \cup \{u_{i}^{Q} : i=1,\ldots,k\},
\end{multline*}
 while the edge set is given by
\begin{eqnarray*}
\hat{E} &=& \{(w_{i}^{T}, x_{j}^{Q}) : (i,j) \in N(T_{A}) \cup N(T_{F}), m(w_{i}^{T}) \neq x_{j}^{Q}\} \\
&& \cup \{(x_{j}^{Q}, w_{i}^{T}) : (i,j) \in N(T_{A}) \cup N(T_{F}), m(w_{i}^{T}) = x_{j}^{Q}\} \\
&& \cup \{(w_{i}^{T}, w_{i}^{Q}) : m(w_{i}^{T}) \neq w_{i}^{Q}\} \cup \{(w_{i}^{Q}, w_{i}^{T}) : m(w_{i}^{T}) = m(w_{i}^{Q})\} \\
&& \cup \{(w_{i}^{T}, u_{j}^{Q}) : (i,j) \in N(T_{B})\} \cup \{(u_{j}^{T}, u_{j}^{Q}) : j=1,\ldots,n\} \\
&& \cup \{(x^{Q},y^{Q}) : x \in \hat{V} \setminus J, y \in J, \tilde{\Omega}_{xy} \neq 0, \tilde{\Omega}_{xz} = 0 \forall z \in J_{1}\}
\end{eqnarray*}

Based on this graph construction, the following sufficient condition for $\mbox{rank}(zF-A|B) = n$ can be derived.

\begin{lemma}
\label{lemma:connectivity_sufficient_general}
Let $$S_{-} = \{v^{Q} : \tilde{\Omega}_{vj} \neq 0 \mbox{ for some } j \in J_{1}\}.$$  Let $V^{\prime\prime}$ denote the set of nodes in $\hat{V}$ that are part of a cycle. If all nodes in $V^{\prime\prime}$ are connected to $S_{-}$ in the graph $\hat{G}$, then the condition $\mbox{rank}(zF-A|B) = n$ holds for almost any values of the free parameters.
\end{lemma}

A proof is given in the appendix.

\subsection{Performance Metrics}
\label{subsec:perf_metrics}
The optimality guarantees for the input selection algorithms presented in this work are applicable to monotone submodular performance metrics.    The first metric is the \emph{network coherence}, defined as follows.

\begin{definition}
\label{def:coherence}
Consider the node dynamics $\dot{x}_{i}(t) = -\sum_{j \in N(i)}{(x_{i}(t)-x_{j}(t))} + w_{i}(t)$, where $w_{i}(t)$ is a zero-mean white process with autocorrelation function $W(\tau) = \delta(\tau)$.  The network coherence $f(S)$ from input set $S$ is the mean-square deviation in the node state from consensus in steady-state.
\end{definition}

The network coherence was defined in \cite{patterson2010leader} and shown to be a supermodular function of the input set in \cite{clark2014leader}.  The second metric is the \emph{convergence error}.

\begin{definition}
\label{def:convergence}
Consider the node dynamics $\dot{x}_{i}(t) = -\sum_{j \in N(i)}{W_{ij}(x_{i}(t)-x_{j}(t))}$, where the edge weights $W_{ij}$ are nonnegative.  The convergence error at time $t$ is defined as $||\mathbf{x}(t) - x^{\ast}\mathbf{1}||_{p}$, for $p \in [1,\infty)$, where $x^{\ast}$ is the state of the input nodes.
\end{definition}

The convergence error was proven to be a supermodular function of the input set in \cite{clark2014convergence}.  Other examples of submodular functions include the information gathered by a set of input nodes \cite{krause2008near} and the trace of the controllability Gramian \cite{summers2014submodularity}.

\section{Problem Formulation - Input Selection for Performance and Controllability}
\label{sec:general}
In this section, we present our submodular optimization framework for selecting input nodes based on performance and controllability. In order to provide computational tractability, we first map the sufficient conditions of Lemma \ref{lemma:condition_s_zero} and Lemma \ref{lemma:connectivity_sufficient_general}  to matroid constraints on the input set.  We present algorithms for selecting a minimum-size input set to guarantee controllability.  We then formulate the problem of selecting a set of up to $k$ input nodes to maximize a performance metric while satisfying controllability.  We prove that the problem is a submodular maximization problem with two matroid basis constraints, and present efficient approximation algorithms.  For the case where the number of input nodes may not be sufficient to guarantee controllability, we introduce a graph controllability index and formulate the problem of selecting input nodes based on a trade-off between performance and controllability.

\subsection{Mapping controllability to matroid constraints}
\label{subsec:controllability_matroid_constraint}
We  derive matroid constraints for the set of non-input nodes that are equivalent to or sufficient for the conditions of Theorem \ref{theorem:matrix_pencil_controllability}.  As a first step, we develop an equivalent representation of the dynamics (\ref{eq:dynamics_input}), and prove that structural controllability of (\ref{eq:dynamics_input}) is equivalent to structural controllability of the equivalent dynamics.

\begin{lemma}
 \label{lemma:equivalent_rep}
 Define the dynamics
  \begin{equation}
 \label{eq:equiv_dynamics}
 F\dot{\mathbf{x}} = \left(
 \begin{array}{c|c}
 A_{RR} & A_{RS} \\
 \hline
 A_{SR}  & A_{SS} 
 \end{array}
 \right)\left(
 \begin{array}{c}
 \mathbf{x}_{R}(t) \\
 \mathbf{x}_{S}(t)
 \end{array}
 \right) + \left(
 \begin{array}{c}
 0 \\
 T_{B}
 \end{array}
 \right)\mathbf{u}(t)
 \end{equation}
 where $T_{B}$ is a diagonal matrix where the diagonal entries are free parameters.  Then the system (\ref{eq:dynamics_input}) is structurally controllable if and only if (\ref{eq:equiv_dynamics}) is structurally controllable.
 \end{lemma}

 \begin{proof}
 Suppose that the system (\ref{eq:equiv_dynamics}) are structurally controllable with $(T_{B})_{ii} = \alpha_{i}$ for some real $\alpha_{i}$'s.
 For a given initial state $\mathbf{x}_{0}$ and desired state $\mathbf{x}^{\ast}$, suppose that there exists a set of control inputs $u_{1}(t),\ldots,u_{k}(t)$ such that $\mathbf{x}(t) = \mathbf{x}^{\ast}$.  We have $$\alpha_{i}u_{i}(t) + \sum_{j}{A_{ij}x_{j}(t)} = \sum_{j}{F_{ij}\dot{x}_{j}(t)},$$ which is equivalent to
 \begin{equation}
 \label{eq:equiv_dynamics_proof}
 \alpha_{i}u_{i}(t) + \sum_{j}{A_{ij}x_{j}(t)} -\sum_{j}{F_{ij}\dot{x}_{j}(t)} + x_{i}(t) = x_{i}(t).
  \end{equation}
  Rearranging terms implies that $x_{i}(t) = \hat{u}_{i}(t)$, where $\hat{u}_{i}(t)$ is the left-hand side of (\ref{eq:equiv_dynamics_proof}).  Hence using $\hat{u}_{i}(t)$ as the input signal implies that structural controllability is achieved for the dynamics (\ref{eq:dynamics_input}) as well.  The proof of the converse is similar.
 \end{proof}

Lemma \ref{lemma:equivalent_rep} implies that it suffices to consider the conditions of Theorem \ref{theorem:matrix_pencil_controllability} under the equivalent system (\ref{eq:equiv_dynamics}).
We first define a matroid constraint on the set of non-input nodes that is equivalent to (\ref{eq:controll2}).  For the system (\ref{eq:equiv_dynamics}), the condition $\mbox{rank}[\mathcal{M}(I|Q_{A}|Q_{B}) \vee \mathcal{M}([I|T_{A}|T_{B}])] = 2n$ of Lemma \ref{lemma:condition_s_zero} is given by $$\mbox{rank}[\mathcal{M}([I|Q_{A}|0]) \vee \mathcal{M}([I|T_{A}|T_{B}(S)])] = 2n,$$ where $T_{B}(S)$ is a diagonal matrix with a free parameter in the $i$-th diagonal entry for $i \in S$ and zeros elsewhere.  In order to establish a matroid constraint for Lemma \ref{lemma:condition_s_zero} we have the following lemma.

\begin{lemma}
\label{lemma:matroid_non_input}
The function $$\rho_{1}(S) = \mbox{rank}(\mathcal{M}([I | Q_{A} | 0]) \vee \mathcal{M}([I | T_{A} | T_{B}(S)])) - \mbox{rank}(\mathcal{M}([I|Q_{A}]) \vee \mathcal{M}([I|T_{A}]))$$ is a matroid rank function.
 \end{lemma}

 \begin{proof}
 First, note that $\rho_{1}(\emptyset) = 0$.  Next, consider $\rho_{1}(S \cup \{v\})$.  Let $\hat{\rho}^{S}$ denote the rank function of $(\mathcal{M}(I|Q_{A}|0) \vee \mathcal{M}([I|T_{A}|T_{B}(S)]))$, and let $\hat{r}_{1}^{S}$ and $\hat{r}_{2}^{S}$ denote the rank functions of $\mathcal{M}([I|Q_{A}|0])$ and $\mathcal{M}([I|T_{A}|T_{B}(S)])$, respectively.  We have $\rho_{1}(S \cup \{v\}) = \hat{\rho}^{S \cup \{v\}}(V)$.  Let $X^{\ast}$ be the set that minimizes $\hat{r}_{1}^{S}(X^{\ast}) + \hat{r}_{2}^{S}(X^{\ast}) + n -|X^{\ast}|$.  Observing that $\hat{r}_{1}^{S} = \hat{r}_{1}^{S \cup \{v\}}$, we have
 \begin{eqnarray*}
 \hat{\rho}^{S \cup \{v\}}(V) &=& \min_{V}{\{\hat{r}_{1}^{S \cup \{v\}}(X) + \hat{r}_{2}^{S \cup \{v\}}(X) + n - |X|\}}\\
 &\leq& 
 \hat{r}_{1}^{S \cup \{v\}}(X^{\ast}) + \hat{r}_{2}^{S \cup \{v\}}(X^{\ast}) + n - |X^{\ast}| \\
  &=& \hat{r}_{1}^{S}(X^{\ast}) + \hat{r}_{2}^{S \cup \{v\}}(X^{\ast}) + n - |X^{\ast}| \\
  &\leq& \hat{r}_{1}^{S}(X^{\ast}) + \hat{r}_{2}^{S}(X^{\ast}) + 1 + n - |X^{\ast}| = \hat{\rho}^{S}(V) + 1 = \rho_{1}(S) +1
  \end{eqnarray*}
 Hence $\rho_{1}(S \cup \{v\}) \leq \rho_{1}(S) + 1$.  On the other hand, for any set $X \subseteq V$,
 \begin{eqnarray*}
 \hat{r}_{1}^{S \cup \{v\}}(X) + \hat{r}_{2}^{S \cup \{v\}}(X) + n - |X| &=& \hat{r}_{1}^{S}(X) + \hat{r}_{2}^{S \cup \{v\}}(X) + n - |X| \\
 &\geq& \hat{r}_{1}^{S}(X) + \hat{r}_{2}^{S}(X) + n - |X|
 \end{eqnarray*}
 and so $\rho_{1}(S \cup \{v\}) \geq \rho_{1}(S)$.  Finally, suppose that $\rho_{1}(S \cup \{v\}) = \rho_{1}(S \cup \{w\}) = \rho_{1}(S)$.  Let $X$ be the set that minimizes $\hat{r}_{1}^{S \cup \{v\}}(X) + \hat{r}_{2}^{S \cup \{v\}}(X) + n - |X|$.  We then have
 \begin{eqnarray*}
 \hat{\rho}^{S \cup \{v,w\}}(X) &\leq& \hat{r}_{1}^{S \cup \{v,w\}}(X) + \hat{r}_{2}^{S \cup \{v,w\}}(X) + n - |X| \\
 &=& \hat{r}_{1}^{S \cup \{v\}}(X) + \hat{r}_{2}^{S \cup \{v,w\}}(X) + n - |X| \\
 &=& \hat{r}_{1}^{S \cup \{v\}}(X) + \hat{r}_{2}^{S \cup \{v\}}(X) + n - |X| = \rho_{1}(S \cup \{v\})
 \end{eqnarray*}
 Since $\rho_{1}(S)$ satisfies the three criteria of Lemma \ref{lemma:matroid_rank}, it is a matroid rank function.
  \end{proof}

 We let $\mathcal{M}_{1}$ denote the matroid with rank function defined by $\rho_{1}(S)$ from Lemma \ref{lemma:matroid_non_input}.  The following corollary allows us to express $\mbox{rank}([A|B]) = n$ as a matroid rank condition on the set of input nodes.
 \begin{corollary}
 \label{corollary:matroid}
 Let $\mathcal{M}_{1}$ be the matroid with rank function defined by $\rho_{1}(S)$ from Lemma \ref{lemma:matroid_non_input}.  Then the rank condition $\mbox{rank}(A|B) = n$ holds for input set $S$ if and only if $\rho_{1}(S) = 2n - \mbox{rank}(\mathcal{M}([I|Q_{A}|0]) \vee \mathcal{M}([I|T_{A}|0]))$. 
 \end{corollary}

 \begin{proof}
 By Lemma \ref{lemma:condition_s_zero}, the condition $\mbox{rank}([A|B]) = n$ is satisfied if and only if $\mbox{rank}(\mathcal{M}([I|Q_{A}|0]) \vee \mathcal{M}([I|T_{A}|T_{B}(S)])) = 2n$.  Combining with the definition of $\rho_{1}(S)$ completes the proof.
 \end{proof}

 Finally, we express $\mbox{rank}([A|B]) = n$ as a matroid constraint on the set of non-input nodes.
 \begin{lemma}
 \label{lemma:s_zero_matroid}
 The  condition $\mbox{rank}([A|B]) = n$ holds if and only if the set of non-input nodes $R$ satisfies $R \in \mathcal{M}_{1}^{\ast}$, the dual of the matroid induced by rank function $\rho_{1}(S)$.
 \end{lemma}

 \begin{proof}
 Let $\rho_{1}^{\ast}$ denote the rank function of the dual matroid $\mathcal{M}_{1}^{\ast}$.  We have $\rho_{1}^{\ast}(R) = \rho_{1}(V \setminus R) + |R| - \rho_{1}(V)$, which is equivalent to $\rho_{1}(S) = \rho_{1}^{\ast}(R) - |R| + \rho_{1}(V)$.  The constraint of Corollary \ref{corollary:matroid} is equivalent to $\rho_{1}(S) \geq \rho_{1}(V)$, which is in turn equivalent to $\rho_{1}^{\ast}(R) \geq |R|$.  This, however, holds only when $R \in \mathcal{M}_{1}^{\ast}$.
 \end{proof}

 We now turn to the constraint that $\mbox{rank}(A-zF|B) = n$ (Eq. (\ref{eq:controll3})). 
   The following intermediate lemma is the first step in our approach.

\begin{lemma}
\label{lemma:equivalent_graph}
There exists a graph $G^{\prime} = (V^{\prime}, E^{\prime})$, which can be constructed in polynomial time, such that the auxiliary graph $\hat{G} = (\hat{V}, \hat{E})$ corresponding to system (\ref{eq:equiv_dynamics}) is given by
\begin{eqnarray}
\label{eq:equiv_aux_graph1}
\hat{V} &=& V^{\prime} \cup \{u_{1}^{T},\ldots,u_{k}^{T}\} \cup \{u_{1}^{Q}, \ldots, u_{k}^{Q}\} \\
\label{eq:equiv_aux_graph2}
\hat{E} &=& E^{\prime} \cup \{(u_{i}^{T}, u_{i}^{Q}) : i=1,\ldots,k\} \cup \{(w_{i}^{T}, u_{i}^{T}) : i \in S\}
\end{eqnarray}
In this graph, the set $J_{1} = \{u_{1}^{Q},\ldots,u_{k}^{Q}\}$ and the condition of Lemma \ref{lemma:connectivity_sufficient_general} is satisfied if each node of $V^{\prime}$ that belongs to a cycle in $G^{\prime}$ is reachable to a node in $\{w_{i}^{T} : i \in S\}$ in the graph $\hat{G}$.
\end{lemma}

\begin{proof}
The matrix $\Omega$ corresponding to (\ref{eq:equiv_dynamics}) is given by
\begin{displaymath}
\Omega = \left(
\begin{array}{cc}
Q_{A}-Q_{F} & 0 \\
I & 0 \\
0 & I
\end{array}
\right)
\end{displaymath}
By solvability of (\ref{eq:equiv_dynamics}) and Lemma \ref{lemma:bipartite_property}, we can select $n$ linearly independent rows $J$ from the first $2n$ rows of $\Omega$. Let $\Psi$ denote the matrix consisting of these linearly independent rows. The matrix $\Omega_{J}$ can be completed to a full-rank matrix by selecting $J_{1} = \{u_{1},\ldots,u_{k}\}$, giving $$\Omega_{J \cup J_{1}} = \left(
\begin{array}{cc}
\Psi & 0 \\
0 & I
\end{array}
\right), \quad
\Omega_{J \cup J_{1}}^{-1} = \left(
\begin{array}{cc}
\Psi^{-1} & 0 \\
0 & I
\end{array}
\right).$$   Note that the matrix $\Omega \Omega_{J \cup J_{1}}^{-1}$ does not depend on the input set $S$.  Let $G^{\prime}$ denote the auxiliary graph when $S = \emptyset$.  By construction, the auxiliary graph $\hat{G}$ with a non-empty input set $S$ is given by (\ref{eq:equiv_aux_graph1}) and (\ref{eq:equiv_aux_graph2}), since adding nodes to $S$ simply adds edges to $N(T_{B})$.

To prove that reachability to the nodes $\{w_{i}^{T} : i \in S\}$ is sufficient, note that it suffices that each node is reachable to $J_{1} = \{u_{1}^{Q}, \ldots, u_{k}^{Q}\}$ by Lemma \ref{lemma:connectivity_sufficient_general}.  Since each node in $\{w_{i}^{T} : i \in S\}$ is reachable to $J_{1}$, it suffices for all other nodes to be reachable to $\{w_{i}^{T} : i \in S\}$.
\end{proof}

As a consequence of Lemmas \ref{lemma:connectivity_sufficient_general} and \ref{lemma:equivalent_graph}, to ensure that $\mbox{rank}(zF - A | B) = n$, it suffices to select an input set $S$ such that each node in $G^{\prime}$ is reachable to $\{w_{i}^{T} : i \in S\}$.
In order to select such an input set, we define an equivalence relation $\sim$ on the nodes in $V^{\prime}$ as $i \sim j$ if node $i$ is path-connected to node $j$ in $G^{\prime}$ and vice versa.  We let $[i] = \{j : i \sim j\}$, and define $\overline{V} = \{[i] : i \in \hat{V}\}$ (so that $\overline{V}$ is the quotient set of $\hat{V}$ under the relation $\sim$).

 Define the graph $\overline{G} = (\overline{V}, \overline{E})$ by $(i,j) \in \overline{E}$ if there exists $i^{\prime} \in [i]$ and $j^{\prime} \in [j]$ such that $(i^{\prime}, j^{\prime}) \in E$.  Note that $\overline{G}$ is a directed acyclic graph.  We let $\overline{V}^{\prime}$ denote the set of isolated nodes, i.e., nodes that have no incoming edges in $\overline{G}$.

 \begin{lemma}
 \label{lemma:connectivity_DAG}
  All nodes in $\hat{V}$ are connected to $S$ iff for each $[i] \in \overline{V}^{\prime}$, $S \cap [i] \neq \emptyset$.
 \end{lemma}

 \begin{proof}
 We first show that if $S \cap [i] \neq \emptyset$ for all $[i] \in \overline{V}^{\prime}$, then all nodes in $\hat{V}$ are connected to $S$.  Let $v \in \hat{V}$.  If $v \in [i]$ with $[i] \in \overline{V}^{\prime}$, then there exists $j \in S \cap [i]$ such that $j$ is path-connected to $v$.  If $v \in [i]$ with $[i] \notin \overline{V}^{\prime}$, then there are nodes $i^{\prime} \in [i]$ and $j_{1} \in [j_{1}]$ for some $[j_{1}] \in \overline{V}$ such that $(j_{1}, i^{\prime}) \in \hat{E}$.  Now, either $[j_{1}] \in \overline{V}^{\prime}$ or there exists $j_{1}^{\prime} \in [j_{1}]$ and $j_{2} \in [j_{2}]$ such that $(j_{2},j_{1}^{\prime}) \in \hat{E}$.  Since the graph $\overline{G}$ is acyclic, there exists a sequence of components $[i], [j_{1}], \ldots, [j_{L}]$, with $([j_{l+1}], [j_{l}]) \in \hat{E}$ and $[j_{L}] \in \overline{V}^{\prime}$.  This sequence of components defines a directed path from a node $v^{\prime} \in S$ to $v$.

 Now, suppose that all nodes in $\hat{V}$ are connected to $S$.  For each $v \in [i]$, with $[i] \in \overline{V}^{\prime}$, $v$ must be connected to at least one input node.  Since $[i] \in \overline{V}^{\prime}$, only other nodes in $[i]$ are connected to $v$.  Hence we must have $S \cap [i] \neq \emptyset$.
 \end{proof}

 Lemma \ref{lemma:connectivity_DAG} enables us to express the connectivity criterion as a matroid constraint.  First, define a function $\rho_{2}(S) = |\{[i] \in \overline{V}^{\prime} : [i] \cap S \neq \emptyset\}|$.  The following lemma describes the rank condition in terms of function $\rho_{2}(S)$.

 \begin{lemma}
 \label{lemma:rho_connectivity}
 Let $c = |\overline{V}^{\prime}|$. The function $\rho_{2}(S)$ is a matroid rank function, and all nodes in $\hat{V}$ are connected to $S$ iff $\rho_{2}(S) = c$.
 \end{lemma}

 \begin{proof}
 The function $\rho_{2}(S)$ is a matroid rank function because $\rho_{2}(\emptyset) = 0$ and $(\rho_{2}(S \cup \{v\}) - \rho_{2}(S)) \in \{0,1\}$, with $\rho_{2}(S \cup \{v\}) - \rho_{2}(S) = 1$ iff there exists $i$ such that $v \in [i]$ and $S \cap [i] = \emptyset$.  Furthermore, $\rho_{2}(S) = c$ if and only if for every $[i] \in \overline{V}^{\prime}$, $S \cap [i] \neq \emptyset$, which is exactly the condition of Lemma \ref{lemma:connectivity_DAG}.
\end{proof}

Let $\mathcal{M}_{2}$ denote the matroid induced by rank function $\rho_{2}(S)$.  We are now ready to state a sufficient matroid constraint on $R$ for the condition (\ref{eq:controll3}).

\begin{lemma}
\label{lemma:connectivity_matroid}
Let $\mathcal{M}_{2}^{\ast}$ be the dual of the matroid induced by $\rho_{2}(S)$. If the non-input nodes $R$ satisfy  $R \in \mathcal{M}_{2}^{\ast}$, then  the condition $\mbox{rank}((zF-A)|B) = n$ is satisfied. 
\end{lemma}

 \begin{proof}
 The rank function $\rho_{2}^{\ast}(R)$ of $\mathcal{M}_{2}^{\ast}$ can be written as $\rho_{2}^{\ast}(R) = \rho_{2}(V \setminus R) + |R| - \rho_{2}(V) = \rho_{2}(S) + |R| - c$, or equivalently, $\rho_{2}(S) = \rho_{2}^{\ast}(R) + c - |R|$.  Hence $\rho_{2}(S) = c$ is equivalent to $\rho_{2}^{\ast}(R) = |R|$, which holds if and only if $R \in \mathcal{M}_{2}^{\ast}$.
 \end{proof}
 
 We combine the results of Lemmas \ref{lemma:s_zero_matroid} and \ref{lemma:connectivity_matroid} to yield the following theorem.
 \begin{theorem}
 \label{theorem:matroid_constraints}
 If $R \in \mathcal{M}_{1}^{\ast} \cap \mathcal{M}_{2}^{\ast}$, then the system is controllable from input set $S = V \setminus R$.
 \end{theorem}

 Having defined matroid-based sufficient conditions for controllability, we will next formulate the problem of selecting the minimum-size input set to guarantee structural controllability.

 \subsection{Minimum-Size Input Set Selection for Structural Controllability}
 \label{subsec:min_controll}
 Selecting a minimum-size set $S$ to satisfy structural controllability is equivalent to selecting a maximum-size set $R=V \setminus S$ that satisfies controllability.  
  Based on Theorem \ref{theorem:matroid_constraints}, the problem of selecting the minimum-size set of input nodes to guarantee structural controllability can be formulated as
  \begin{equation}
  \label{eq:minimum_input}
  \mbox{maximize}\left\{|R|: R \in \mathcal{M}_{1}^{\ast}, R \in \mathcal{M}_{2}^{\ast}\right\}.
  \end{equation}

  \begin{lemma}
  \label{lemma:min_input_polynomial}
  A minimum-size set of input nodes satisfying structural controllability can be obtained in polynomial-time.
  \end{lemma}

  \begin{proof}
  The problem of selecting a minimum-size set of input nodes to satisfy structural controllability is formulated as (\ref{eq:minimum_input}).  Eq. (\ref{eq:minimum_input}) is a matroid intersection problem, which can be solved in time $O(n^{5/2}\tau)$, where $\tau$ is the time required to test if a set $R$ is in $\mathcal{M}_{1}^{\ast}$ and $\mathcal{M}_{2}^{\ast}$ \cite{schrijver2003combinatorial}. The independence of $R$ in each set can be evaluated in polynomial time.
  \end{proof}


  Algorithm \ref{algo:min_set} gives a polynomial-time procedure for solving (\ref{eq:minimum_input}) using the maximum cardinality matroid intersection algorithm of \cite[Ch. 41]{schrijver2003combinatorial}.   Lemmas \ref{lemma:matroid_non_input} and \ref{lemma:connectivity_matroid} and the above discussion generalize the main result of \cite{liu2011controllability} from free matrices to systems with a mix of free and fixed parameters.

        \begin{center}
\begin{algorithm}[!ht]
	\caption{Algorithm for selecting the minimum-size input set to guarantee structural controllability.}
	\label{algo:min_set}
	\begin{algorithmic}[1]
		\Procedure{Min\_Controllable\_Set}{$\mathcal{M}_{1}^{\ast}$, $\mathcal{M}_{2}^{\ast}$}
        \State \textbf{Input:} Matroids $\mathcal{M}_{1}^{\ast}$ and $\mathcal{M}_{2}^{\ast}$
        \State \textbf{Output:} Set of inputs $S$
        \State $R \leftarrow \emptyset$
        \While{1}
        \State $E_{\mathcal{M}_{1}^{\ast}, \mathcal{M}_{2}^{\ast}}(R) \leftarrow \emptyset$
        \For{All $i \in R$, $j \notin R$}
        \If{$(R-\{i\}\cup\{j\}) \in \mathcal{M}_{1}^{\ast}$}
        \State $E_{\mathcal{M}_{1}^{\ast}, \mathcal{M}_{2}^{\ast}}(R) \leftarrow E_{\mathcal{M}_{1}^{\ast}, \mathcal{M}_{2}^{\ast}}(R) \cup \{(i,j)\}$
        \EndIf
        \If{$(R-\{i\} \cup \{j\}) \in \mathcal{M}_{2}^{\ast}$}
        \State $E_{\mathcal{M}_{1}^{\ast}, \mathcal{M}_{2}^{\ast}}(R) \leftarrow E_{\mathcal{M}_{1}^{\ast}, \mathcal{M}_{2}^{\ast}}(R) \cup \{(j,i)\}$
        \EndIf
        \EndFor
        \State $D_{\mathcal{M}_{1}^{\ast}, \mathcal{M}_{2}^{\ast}}(R) \leftarrow$ directed graph with vertex set $V$ and edge set $E_{\mathcal{M}_{1}^{\ast}, \mathcal{M}_{2}^{\ast}}(R)$
        \State $X_{1} \leftarrow \{j \in V \setminus R : (R \cup \{j\}) \in \mathcal{M}_{1}^{\ast}\}$
        \State $X_{2} \leftarrow \{j \in V \setminus R : (R \cup \{j\}) \in \mathcal{M}_{2}^{\ast}\}$
        \If{path exists from a node in $X_{1}$ to a node in $X_{2}$}
        \State $P \leftarrow$ shortest $X_{1}$-$X_{2}$ path
        \State $R \leftarrow R \Delta P$
        \Else
        \State \textbf{break}
        \EndIf
        \EndWhile
        \State $S \leftarrow V \setminus R$
        \State \Return{$S$}
        \EndProcedure
        \end{algorithmic}
        \end{algorithm}
        \end{center}

  \subsection{Input Selection for Joint Performance and Controllability}
  \label{subsec:joint_perf_controll}
  We now consider the problem of maximizing a monotone performance metric $f(S)$ while satisfying controllability with a set of up to $k$ input nodes.  Based on Theorem \ref{theorem:matroid_constraints}, the problem formulation is given by
  \begin{equation}
  \label{eq:joint_perf_controll}
  \begin{array}{ll}
  \mbox{maximize}_{S \subseteq V} & f(S) \\
  \mbox{s.t.} &  |S| \leq k \\
   & (V \setminus S) \in \mathcal{M}_{1}^{\ast} \\
   & (V \setminus S) \in \mathcal{M}_{2}^{\ast}
   \end{array}
   \end{equation}

   Eq. (\ref{eq:joint_perf_controll}) is a combinatorial optimization problem, making it NP-hard to solve in the general case.  The following lemma describes an equivalent formulation to (\ref{eq:joint_perf_controll}).

   \begin{lemma}
   \label{lemma:joint_perf_controll_equiv}
   Define $r_{1} = \mbox{rank}(\mathcal{M}_{1})$ and $r_{2} = \mbox{rank}(\mathcal{M}_{2})$.
   Let $\hat{\mathcal{M}}_{1} = \mathcal{M}_{1} \vee U_{k-r_{1}}$ and $\hat{\mathcal{M}}_{2} = \mathcal{M}_{2} \vee U_{k-r_{2}}$, with $\hat{\mathcal{B}}_{1}$ and $\hat{\mathcal{B}}_{2}$ denoting the sets of bases of $\hat{\mathcal{M}}_{1}$ and $\hat{\mathcal{M}}_{2}$, respectively.  Let $S^{\ast}$ denote the optimal solution to the problem
   \begin{equation}
   \label{eq:joint_perf_controll_equiv}
   \begin{array}{ll}
   \mbox{maximize}_{S \subseteq V} & f(S) \\
   \mbox{s.t.} & S \in \hat{\mathcal{B}}_{1} \cap \hat{\mathcal{B}}_{2}
   \end{array}
   \end{equation}
   The set $S^{\ast}$ is the optimal solution to (\ref{eq:joint_perf_controll}).
   \end{lemma}

   \begin{proof}
   First, we have that the optimal solution to (\ref{eq:joint_perf_controll}) satisfies $|S| = k$.  If not, then since $f(S)$ is monotone, we can add elements to $S$ and increase the value of $f$ without violating the constraints $(V \setminus S) \in \mathcal{M}_{1}^{\ast}$ and $(V \setminus S) \in \mathcal{M}_{2}^{\ast}$.  Furthermore, if $R=(V \setminus S) \in \mathcal{M}_{1}^{\ast}$, then $R$ can be completed to a basis $\overline{R}$ of $\mathcal{M}_{1}^{\ast}$, and we have $(V \setminus \overline{R}) \subseteq (V \setminus R)$, implying that $S = (V \setminus \overline{R}) \cup \overline{S}$ for some set $\overline{S}$ with $|\overline{S}| = k-r_{1}$.  Since $\overline{R}$ is a basis of $\mathcal{M}_{1}^{\ast}$, $V \setminus \overline{R}$ is a basis of $\mathcal{M}_{1}$, and hence $S \in \hat{\mathcal{B}}_{1}$.  A similar result holds for $\hat{\mathcal{M}}_{2}$.

     Now, consider $S^{\ast} \in \hat{\mathcal{B}}_{1} \cap \hat{\mathcal{B}}_{2}$. By construction $\mbox{rank}(\hat{\mathcal{M}}_{1}) = \mbox{rank}(\hat{\mathcal{M}}_{2}) = k$.  We can therefore write $S^{\ast} = S_{1}^{\ast} \cup \hat{S}_{1}^{\ast}$, where $S_{1}^{\ast}$ is a basis of $\mathcal{M}_{1}$ and $|\hat{S}_{1}^{\ast}| = k-r_{1}$.  Since $S_{1}^{\ast}$ is a basis of $\mathcal{M}_{1}$, we have $(V \setminus S_{1}^{\ast}) \in \mathcal{M}_{1}^{\ast}$, and thus the set $R = (V \setminus S^{\ast}) \subseteq (V \setminus S_{1}^{\ast}) \in \mathcal{M}_{1}^{\ast}$, implying that $R = (V \setminus S^{\ast}) \in \mathcal{M}_{1}^{\ast}$.  A similar argument implies that $(V \setminus S^{\ast}) \in \mathcal{M}_{2}^{\ast}$.
   \end{proof}

   By Lemma \ref{lemma:joint_perf_controll_equiv}, solving (\ref{eq:joint_perf_controll}) is equivalent to solving (\ref{eq:joint_perf_controll_equiv}).  We present a two-stage algorithm for approximating (\ref{eq:joint_perf_controll_equiv}).  In the first stage, the algorithm solves a relaxed, continuous version of the problem.  In the second stage, the algorithm rounds the solution to an integral value satisfying the constraints of (\ref{eq:joint_perf_controll_equiv}).   

   The algorithm is defined as Algorithm \ref{algo:Joint_PC} below.  It contains two subroutines, namely, MAX\_WEIGHTED\_BASIS and SWAP\_ROUND.  The subroutine \\
    MAX\_WEIGHTED\_BASIS takes as input two matroids $\mathcal{M}^{\prime}$ and $\mathcal{M}^{\prime\prime}$ with the same ground set $V$ (with $|V| = n$), as well as a weight vector $\boldsymbol{\alpha} \in \mathbb{R}^{n}$, and outputs a set $I \in \mathcal{B}(\mathcal{M}^{\prime}) \cap \mathcal{B}(\mathcal{M}^{\prime\prime})$ such that $\sum_{i \in I}{\alpha_{i}}$ is maximized (provided at least one common basis exists).  Polynomial-time algorithms for finding such sets are well-known~\cite[Ch. 43]{schrijver2003combinatorial}.

   The subroutine SWAP\_ROUND takes as input a vector $\mathbf{r}$ in the common base polytope of two matroids $\mathcal{M}^{\prime}$ and $\mathcal{M}^{\prime\prime}$, and outputs a set $I \in \mathcal{B}(\mathcal{M}^{\prime}) \cap \mathcal{B}(\mathcal{M}^{\prime\prime})$.  The algorithm is randomized with the output satisfying $\mathbf{E}(f(I)) \geq F(\mathbf{r})$.  The swap round algorithm was proposed in \cite{chekuri2010dependent}.

      \begin{center}
\begin{algorithm}[!ht]
	\caption{Input selection algorithm for joint performance and controllability.}
	\label{algo:Joint_PC}
	\begin{algorithmic}[1]
		\Procedure{Input\_Select}{$f$, $\hat{\mathcal{M}}_{1}$, $\hat{\mathcal{M}}_{2}$, $k$}
        \State \textbf{Input:} Monotone submodular objective function $f: 2^{V} \rightarrow \mathbb{R}$
        \State \quad Matroids $\hat{\mathcal{M}}_{1}$, $\hat{\mathcal{M}}_{2}$
        \State \quad Maximum number of inputs $k$
        \State \textbf{Output:} Set of inputs $S$
        \State $\delta \leftarrow \frac{1}{9k^{2}}$, $t \leftarrow 0$, $\mathbf{y}(0) \leftarrow \mathbf{0}$
        \While{$t < 1$}
        \State $R(t)$ contains each $j \in V$ independently with probability $y_{j}(t)$
        \For{$j \in V$}
        \State $\omega_{j}(t) \leftarrow \mathbf{E}[f(R(t) \cup \{j\}) - f(R(t))]$
        \EndFor
        \State $I(t) \leftarrow \mbox{MAX\_WEIGHTED\_BASIS}(\hat{\mathcal{M}}_{1}, \hat{\mathcal{M}}_{2}, \boldsymbol{\omega})$
        \State $\mathbf{y}(t + \delta) \leftarrow \mathbf{y}(t) + \delta\cdot\mathbf{1}(I(t))$
        \State $t \leftarrow (t+\delta)$
        \EndWhile
        \State $S \leftarrow \mbox{SWAP\_ROUND}(\mathbf{y}(1), \hat{\mathcal{M}}_{1}, \hat{\mathcal{M}}_{2})$
        \State \Return{$S$}
        \EndProcedure
        \end{algorithmic}
        \end{algorithm}
        \end{center}

        In Algorithm \ref{algo:Joint_PC}, the $\mathbf{1}(I(t))$ denotes the incidence vector of set $I(t)$, which has a $1$ in the $i$-th entry if $i \in I$ and $0$ otherwise. The following theorem describes the optimality bound of Algorithm \ref{algo:Joint_PC}.

        \begin{theorem}
        \label{theorem:Joint_PC}
        Algorithm \ref{algo:Joint_PC} runs in polynomial time with complexity $O(\tau n^{5})$.  Letting $S^{\ast}$ denote the optimal solution to (\ref{eq:joint_perf_controll}), the vector $\mathbf{y}(1)$ returned by the continuous relaxation satisfies $F(\mathbf{y}(1)) \geq (1-1/e)f(S^{\ast})$, where $F$ is the multilinear relaxation of $f(S)$. The rounded solution $S$ is a feasible solution to (\ref{eq:joint_perf_controll}). 
        \end{theorem}

        The proof is given in the appendix. 
        The following theorem provides additional optimality guarantees when the objective function $f(S)$ is linear.
        \begin{theorem}
        \label{theorem:joint_PC_linear}
        If the function $f(S)$ is of the form $f(S) = \sum_{i \in S}{\tau_{i}}$ for some real-valued weights $\tau_{1},\ldots,\tau_{n}$, then the solution to (\ref{eq:joint_perf_controll}) can be obtained in polynomial time.
        \end{theorem}
        \begin{proof}
        If the function $f(S)$ is of the form $f(S) = \sum_{i \in S}{\tau_{i}}$, then (\ref{eq:joint_perf_controll}) is equivalent to maximizing a modular function subject to two matroid basis constraints.  For problems of this form, the Edmonds weighted matroid intersection algorithm provides an optimal solution in polynomial time \cite{schrijver2003combinatorial}.
        \end{proof}

        \subsection{Selecting Input Nodes for Performance-Controllability Trade-Off}
        \label{subsec:GCI}
        In this section, we study input selection based on a trade-off between performance and controllability, instead of treating controllability as a constraint that must be satisfied.  This maximization may be beneficial when the  number of input nodes $k$ is insufficient to guarantee controllability.

        We first introduce two graph controllability indices (GCIs) that can be traded off with a performance metric in order to maximize the level of performance and controllability.  The first controllability index $c_{1}(S)$ is given by
        \begin{equation}
        \label{eq:GCI1}
        c_{1}(S) = \max{\{|V^{\prime}| : \mbox{rank}(A(V^{\prime}) | B(V^{\prime})) = |V^{\prime}|\}},
        \end{equation}
        where $A(V^{\prime})$ and $B(V^{\prime})$ are sub-matrices of $A$ and $B$ consisting of the rows and columns indexed in $V^{\prime}$.  
        Intuitively, $c_{1}(S)$ is the size of the largest subgraph of $V$ such that the zero modes of all nodes in the subgraph are controllable.  The second GCI quantifies the controllability of the nonzero modes, characterized by the constraint $\mbox{rank}(zF-A | B) = n$.  We define $c_{2}(S)$ by 
        \begin{equation}
        \label{eq:GCI2}
        c_{2}(S) = |\{i \in V : i \mbox{ is reachable to $\mathbf{u}$ in $\hat{G}$}\}|
        \end{equation}
        where $\hat{G}$ is defined as in Section \ref{subsec:aux_graph}.  The function $c_{2}(S)$ quantifies the number of nodes that are reachable to the input nodes, and hence satisfy controllability of the nonzero modes.  If $c_{1}(S) + c_{2}(S) = 2n$, then both the zero and nonzero modes of all nodes are controllable, and hence controllability is satisfied.  Otherwise, the problem of joint maximization of performance and controllability can be formulated as
        \begin{equation}
        \label{eq:GCI_opt}
        \begin{array}{ll}
        \mbox{maximize}_{S \subseteq V} & f(S) + \eta (c_{1}(S) + c_{2}(S)) \\
        \mbox{s.t.} & |S| \leq k
        \end{array}
        \end{equation}

       The trade-off parameter $\eta \geq 0$ is used to vary the relative weight assigned to performance or controllability criteria.  When $\eta$ is small, then nodes are selected for performance alone; at the other extreme, when $\eta$ is large, nodes are primarily selected to maximize controllability. The following result is the first step in deriving efficient algorithms for solving (\ref{eq:GCI_opt}).

        \begin{theorem}
        \label{theorem:GCI_submodular}
        The functions $c_{1}(S)$ and $c_{2}(S)$ are submodular as functions of $S$.
        \end{theorem}

        \begin{proof}
        The function $c_{1}(S)$ is equal to the maximum-size set of non-input nodes with controllable zero modes from the input nodes, plus the number of input nodes.  This can be written as $c_{1}(S) = \rho_{1}(V\setminus S) + |S|$, where $\rho_{1}$ is defined as in Lemma \ref{lemma:matroid_non_input}. Since $\rho_{1}$ is a matroid rank function, $\rho_{1}$ is submodular and hence $\rho_{1}(V \setminus S)$ is submodular as well by Lemma \ref{lemma:submod_comp}.  Since the sum of submodular functions is submodular, $\rho_{1}(V \setminus S) + |S|$ is submodular as a function of $S$.

        It remains to show submodularity of $c_{2}(S)$. Let $S \subseteq T$, and suppose that $v \notin T$.  We have that $c_{2}(T \cup \{v\}) - c_{2}(T)$ is equal to the number of nodes that are reachable to $v$, but not to any node in $T$.  Since $S \subseteq T$, any node that is not reachable to $T$ is automatically not reachable to $S$.  Hence, any node that is reachable to $v$ but not any node in $T$ is also reachable to $v$ but not any node in $S$, implying that $c_{2}(T \cup \{v\}) - c_{2}(T) \geq c_{2}(S \cup \{v\}) - c_{2}(S)$.
        \end{proof}


        A greedy algorithm for approximating (\ref{eq:GCI_opt}) is as follows.  The set $S$ is initialized to be empty, and the algorithm proceeds over $k$ iterations.  At the $i$-th iteration, the element $v \in V$ maximizing $f(S \cup \{v\}) + \eta (c_{1}(S \cup \{v\}) + c_{2}(S \cup \{v\})$ is selected and added to $S$, terminating after $k$ iterations.   The following theorem gives an optimality bound for this algorithm.
        \begin{theorem}
        \label{theorem:GCI_optimal}
        The solution $S$ obtained by the greedy algorithm satisfies $(f(S) + \eta(c_{1}(S) + c_{2}(S)))  \geq (1-1/e)(f(S^{\ast}) + \eta(c_{1}(S^{\ast}) + c_{2}(S^{\ast})))$, where $S^{\ast}$ is the optimal solution to (\ref{eq:GCI_opt}).
        \end{theorem}

        \begin{proof}
        Since $f(S)$, $c_{1}(S)$, and $c_{2}(S)$ are submodular and monotone as functions of $S$, the function $f(S) + \eta(c_{1}(S) + c_{2}(S))$ is monotone and submodular.  Hence, Theorem 4.1 of \cite{nemhauser1978analysis} implies that the greedy algorithm returns a set $S$ satisfying a $(1-1/e)$-optimality bound with the optimal set $S^{\ast}$, thus completing the proof.  
        \end{proof} 
\section{Input Selection in Strongly Connected Networks}
\label{sec:connected}
The input selection algorithms of the previous section hold for any arbitrary structured linear descriptor system.  In this section, we investigate the case where the graph induced by the system matrices $F$, $A$, and $B$ is strongly connected, i.e., there exists a directed path from any node $i$ to any node $j$.  In this case, there is additional problem structure that reduces the complexity and improves the optimality bounds of our input selection algorithms.  The following lemma gives system properties that hold with high probability for strongly connected networks.

\begin{lemma}
\label{lemma:strongly_conn_condition}
If the graph induced by $(F,A,B)$ is strongly connected and \\
$\mbox{rank}(A|B) = n$, then the condition $\mbox{rank}((zF-A)|B) = n$ holds for almost any fixed parameter matrices $Q_{F}$, $Q_{A}$, and $Q_{B}$.
\end{lemma}

\begin{proof}
If $\mbox{rank}(A|B) = n$, then $\det{(zF-A)} \neq 0$ for all $z \in \mathbb{C}$, except for some complex numbers $z_{1},\ldots,z_{n}$. Consider $z_{i}$ such that $\det{(z_{i}F-A)} = 0$.  We have that $$\det{(z_{i}F-A)} = \sum_{\sigma \in S_{n}}{\prod_{j=1}^{n}{(z_{i}F-A)_{j\sigma(j)}}}.$$ We observe that each $\sigma$ corresponding to a nonzero term of the summation induces a decomposition of the graph into cycles $j_{1},\ldots,j_{m}$, in which $j_{l+1} = \sigma(j_{l})$ and $j_{1} = \sigma(j_{m})$.  If the determinant is zero, then there exist at least two such decompositions, corresponding to distinct permutations $\sigma$ and $\sigma^{\prime}$, with products of weights that sum to zero.   

Suppose that the $l$-th column of the matrix $(z_{i}F-A)$ is linearly dependent on the other columns.  Suppose that the column is replaced by one of the input columns from $B$.  Since the graph is strongly connected, a new cycle is induced by adding the input column.  For almost all values of the free parameters of $B$, the cycle will not cancel out with the other cycles in the graph, and hence the determinant will be nonzero.
\end{proof}

  If $\mbox{rank}((zF-A)|B) = n$, then only the constraint $(V \setminus S) \in \mathcal{M}_{1}^{\ast}$ must hold.  We now describe how this additional problem structure improves the runtime and optimality gaps of each of the input selection problems considered.

  \subsection{Minimum-Size Input Set Selection in Strongly Connected Networks}  In the strongly connected case, the minimum-size input set selection problem reduces to
\begin{equation}
\label{eq:min_input_connected}
\begin{array}{ll}
\mbox{maximize} & |R| \\
\mbox{s.t.} & R \in \mathcal{M}_{1}^{\ast}
\end{array}
\end{equation}

The following algorithm can be used to compute the solution to (\ref{eq:min_input_connected}).  Initialize the set $R = \emptyset$, and let $V = \{1,\ldots,n\}$ be the set of possible input nodes.  The algorithm iterates over all nodes in $V$, starting with the node indexed $1$. For each node $i$, test if $(R \cup \{i\}) \in \mathcal{M}_{1}^{\ast}$.  If so, set $R = R \cup \{i\}$.  The algorithm terminates after all $n$ nodes have been tested.


\begin{lemma}
\label{lemma:greedy_min_connected}
When the network is strongly connected and the condition of Lemma \ref{lemma:strongly_conn_condition} holds, the greedy algorithm returns the minimum-size input set to guarantee controllability within $O(n)$ computations of the matroid independence condition $R \in \mathcal{M}_{1}^{\ast}$.
\end{lemma}

\begin{proof}
Let $\mathcal{B}$ be the set of bases of $\mathcal{M}_{1}^{\ast}$.  We  define a lexicographic ordering on the sets $R \in \mathcal{B}$ as follows.  Let $R_{1}$, $R_{2} \in \mathcal{B}$, and write $R_{1} = \{a_{1},\ldots,a_{m}\}$ and $R_{2} = \{b_{1},\ldots,b_{m}\}$, where $a_{1} < a_{2} < \cdots < a_{m}$ and $b_{1} < b_{2} < \cdots < b_{m}$.  We let $R_{1} \prec R_{2}$ if there exists $i$ such that $a_{j} = b_{j}$ for $j < i$ and $a_{i} < b_{i}$.  We have that $\prec$ induces a total ordering on $\mathcal{B}$, since for any $R_{1}$ and $R_{2}$ with $R_{1} \neq R_{2}$, we have $R_{1} \prec R_{2}$ or $R_{2} \prec R_{1}$.

Let $R^{\ast}$ denote the set in $\mathcal{B}$ that is minimal under the ordering $\prec$.  We show that the algorithm described above outputs $R^{\ast}$.  Let $R^{\ast} = \{a_{1},\ldots,a_{m}\}$, and let $R_{i}^{\ast} = R^{\ast} \cap \{1,\ldots,i\}$.  Finally, let $R_{i}$ denote the set computed by our algorithm at iteration $i$.  We prove by induction that $R_{i}^{\ast} = R_{i}$ for each $i$.

We have $R_{0} = R_{0}^{\ast}$ trivially.  Now, suppose $R_{i} = R_{i-1}^{\ast}$.  We have two cases.  First, suppose that $i \in R^{\ast}$.  Since $R_{i-1} \cup \{i\} = R_{i-1}^{\ast} \cup \{i\} = R_{i}^{\ast}$ and $R_{i}^{\ast} \subseteq R^{\ast}$, we have that $(R_{i-1} \cup \{i\}) \in \mathcal{M}_{1}^{\ast}$.  Hence the algorithm will add $i$ to the set, and $R_{i} = R_{i}^{\ast}$.

Now, suppose that $i \notin R^{\ast}$, and suppose that $R_{i}^{\ast} \neq R_{i}$.  By inductive hypothesis, we must have that $i \in R_{i}$, which occurs if and only if $R_{i} = (R_{i-1} \cup \{i\})$ is independent in $\mathcal{M}_{1}^{\ast}$.  Since $\mathcal{M}_{1}^{\ast}$ is a matroid, we can complete $R_{i}$ to a basis $\hat{R} \in \mathcal{B}$.  By definition of the $\prec$ ordering, $\hat{R} \prec R^{\ast}$, contradicting the assumption that $R^{\ast}$ is minimal under the ordering $\prec$.  This contradiction implies that $i \notin R_{i}$, and so $R_{i} = R_{i}^{\ast}$.

Continuing inductively until $i=n$, we have that $R_{n} = R_{n}^{\ast} = R^{\ast}$.  Since $R_{n}$ is equal to $R^{\ast} \in \mathcal{B}$, $R_{n}$ is a basis of $\mathcal{M}_{1}^{\ast}$, and hence is a solution to (\ref{eq:min_input_connected}).  The $O(n)$ runtime follows from the fact that the algorithm makes one independence check per iteration over a total of $n$ iterations.
\end{proof}


The additional problem structure in the strongly connected case leads to a simplified algorithm with reduced runtime compared to Algorithm \ref{algo:min_set}.


\subsection{Joint Performance and Controllability Input Selection in Strongly Connected Networks}
\label{subsec:Joint_PC_strong_conn}
When the system graph is strongly connected and the conditions of Lemma \ref{lemma:strongly_conn_condition} hold, the complexity and optimality bounds of input selection for joint performance and controllability are improved.  With this additional structure, the problem formulation is given by
\begin{equation}
\label{eq:joint_PC_strong_conn}
\begin{array}{ll}
\mbox{maximize} & f(S) \\
\mbox{s.t.} & |S| \leq k \\
 & (V \setminus S) \in \mathcal{M}_{1}^{\ast}
 \end{array}
 \end{equation}
 The following lemma gives an equivalent formulation to (\ref{eq:joint_PC_strong_conn}).
 \begin{lemma}
 \label{lemma:PC_strong_conn_equiv}
 Let $r_{1} = \mbox{rank}(\mathcal{M}_{1})$ and define $\hat{\mathcal{M}}_{1} = \mathcal{M}_{1} \vee U_{k-r_{1}}$.  If the objective function $f(S)$ is monotone, then the optimization problem (\ref{eq:joint_PC_strong_conn}) has the same solution as
 \begin{equation}
 \label{eq:joint_PC_SC_equiv}
 \begin{array}{cc}
 \mbox{maximize} & f(S) \\
 \mbox{s.t.} & S \in \hat{\mathcal{M}}_{1}
 \end{array}
 \end{equation}
 \end{lemma}

 \begin{proof}
 Let $S^{\ast}$ and $\hat{S}$ denote the optimal solutions to (\ref{eq:joint_PC_strong_conn}) and (\ref{eq:joint_PC_SC_equiv}), respectively.  We observe that both $|S^{\ast}| = |\hat{S}| = k$ by monotonicity of $f(S)$.  We show that if $|S| = k$, then the conditions $(V \setminus S) \in \mathcal{M}_{1}^{\ast}$ and $S \in \hat{\mathcal{M}}_{1}$ are equivalent.  First, suppose that $(V \setminus S) \in \mathcal{M}_{1}^{\ast}$.  Let $R^{\ast} = (V \setminus S) \in \mathcal{M}_{1}^{\ast}$.  The set $R^{\ast}$ can be completed to a basis of $\mathcal{M}_{1}^{\ast}$, denoted $\hat{R} = R^{\ast} \cup R^{\prime}$, and so we have $S = (V \setminus \hat{R}) \cup (V \setminus R^{\prime})$.  Now, $(V \setminus \hat{R}) \in \mathcal{M}_{1}$ and $|V \setminus R^{\prime}| = k-r_{1}$, and so $S \in \hat{\mathcal{M}}_{1}$.

 Suppose that $S \in \hat{\mathcal{M}}_{1}$ and $|S| = k$.  Since $|S| = k$, $S$ is a basis of $\hat{\mathcal{M}}_{1}$, and so $S$ can be written as $S = S_{1} \cup S_{2}$ where $S_{1} \in \mathcal{M}_{1}$.  Hence $(V \setminus S) \subseteq (V \setminus S_{1}) \in \mathcal{M}_{1}^{\ast}$, and so the second constraint of (\ref{eq:joint_PC_strong_conn}) is satisfied.
  \end{proof}


Convex relaxation approaches have been proposed for solving matroid-constrained monotone submodular maximization problems~\cite{calinescu2011maximizing,chekuri2010dependent}. One such approach is to replace Line 12 in Algorithm \ref{algo:Joint_PC} with a subroutine that computes the maximum-weighted basis of a matroid (such a basis can be computed efficiently using a greedy algorithm).  It was shown in \cite{calinescu2011maximizing} that this algorithm achieves a $(1-1/e)$ optimality bound.   


\subsection{Performance-Controllability Trade-Off in Strongly Connected Networks}
\label{subsec:SC_GCI}
 In the strongly connected network case, we define the graph controllability index (GCI) $$c(S) = \max{\{|V^{\prime}| : V^{\prime} \mbox{ controllable from $S$}\}}.$$  The following lemma provides additional structure on $c(S)$.

 \begin{lemma}
 \label{lemma:GCI_matroid_rank}
 The function $c(S)= \tilde{c}(S) + \zeta$, where $\zeta$ is a constant and $\tilde{c}(S)$ is a matroid rank function.
 \end{lemma}

 \begin{proof}
 The largest controllable subgraph of $G$ corresponds to a subset of states such that, for the matrix $A^{\prime}$ with columns indexed in $V^{\prime}$, we have $\mbox{rank}(\mathcal{M}([I|Q_{A^{\prime}}|0] \vee [I|T_{A^{\prime}}|T_{B}(S)]) = 2|V^{\prime}|$.  This subset of columns, however, is exactly the maximum-size independent set in $(\mathcal{M}([I|Q_{A}|0]) \vee \mathcal{M}([I|T_{A}|T_{B}(S)]))$, and the value of $c(S)$ is $\mbox{rank}(\mathcal{M}([I|Q_{A}|0]) \vee \mathcal{M}([I|T_{A}|T_{B}(S)]))$.  By Lemma \ref{lemma:matroid_non_input}, the rank is equal to the rank of $\mathcal{M}([I|Q_{A}]) \vee \mathcal{M}([I|T_{A}])$ plus a matroid rank function of $S$.
 \end{proof}

 The problem of selecting a set of up to $k$ input nodes to maximize both a performance metric $f(S)$ and the GCI $c(S)$ is formulated as
 \begin{equation}
 \label{eq:SC_GCI_form}
 \begin{array}{ll}
 \mbox{maximize} & f(S) + \eta c(S) \\
 \mbox{s.t.} & |S| \leq k
 \end{array}
 \end{equation}
 As in the general case, a greedy algorithm for maximizing $f(S) + \eta c(S)$ is guaranteed to return an input set $S^{\ast}$ such that $f(S^{\ast}) + \eta c(S^{\ast})$ is within a $(1-1/e)$ factor of the optimum.  Moreover, when the performance metric $f(S)$ is identically zero, so that only controllability is optimized, we have the following result.
 \begin{lemma}
 \label{lemma:SC_GCI_opt}
 If $f(S) = 0$, then the greedy algorithm returns the optimal solution to (\ref{eq:SC_GCI_form}).
 \end{lemma}

 \begin{proof}
 For the problem of maximizing a matroid rank function subject to a cardinality constraint, the greedy algorithm is known to return an optimal solution \cite{oxley1992matroid}.  If $f(S) = 0$, then by Lemma \ref{lemma:GCI_matroid_rank}, Eq. (\ref{eq:SC_GCI_form}) is equivalent to maximizing a matroid rank function subject to a cardinality constraint, and hence the greedy algorithm returns the optimal input set $S$.   
 \end{proof} 
\section{Special Cases of Our Approach}
\label{sec:cases}
In this section, we consider three special cases of our framework, namely systems with linear consensus dynamics, networked systems where each node has second integrator dynamics, and systems where all parameters are free.

\subsection{Linear Consensus Dynamics}
\label{subsec:consensus}
We first consider a network of $N$ nodes where each node $i \in \{1,\ldots,N\}$ has a state $x_{i}(t) \in \mathbb{R}$.  The state dynamics of the non-input nodes are given by $\dot{x}_{i}(t) = -\sum_{j \in N(i)}{W_{ij}(x_{i}(t)-x_{j}(t))}$, where $W_{ij}$ are nonnegative weights.  In \cite{goldin2013weight}, it was shown that, by introducing a set of states $\{x_{j}^{e} : j =1,\ldots,M\}$, where $M$ is the number of edges in the network, the system can be written in the form (\ref{eq:LDS}) as
\begin{equation}
\label{eq:consensus_LDS}
\left(
\begin{array}{cc}
I & 0 \\
0 & 0
\end{array}
\right)\left(
\begin{array}{c}
\dot{\mathbf{x}}(t) \\
\dot{\mathbf{x}}^{e}(t)
\end{array}
\right) = \left(
\begin{array}{cc}
0 & K \\
K_{I} & W
\end{array}
\right)\left(
\begin{array}{c}
\mathbf{x}(t) \\
\mathbf{x}^{e}(t)
\end{array}
\right).
\end{equation}
 In (\ref{eq:consensus_LDS}), $K_{I}$ is the incidence matrix of the graph and $K$ is the transpose of the incidence matrix. $W$ is a diagonal matrix with $e$-th entry equal to the weight on edge $e$. We assume that the weights $W$ are free parameters, so that $Q_{F}$, $T_{F}$, $Q_{A}$, and $T_{A}$ are given by
 \begin{equation}
 Q_{F} = \left(
 \begin{array}{cc}
 I & 0 \\
 0 & 0
 \end{array}
 \right), \quad Q_{A} = \left(
 \begin{array}{cc}
 0 & K \\
 K_{I} & 0
 \end{array}
 \right), \quad T_{A} = \left(
 \begin{array}{cc}
 0 & 0 \\
 0 & W
 \end{array}
 \right).
 \end{equation}
 As a first step towards analyzing this class of system dynamics under our framework, we consider the condition of Lemma \ref{lemma:connectivity_sufficient_general}.  For this system, the matrix $\Omega$ of Section \ref{subsec:aux_graph} is equal to
 \begin{equation}
 \label{eq:consensus_Omega}
 \Omega =
 \begin{array}{c}
 \mathbf{w} \\
 \mathbf{w}^{e} \\
 \mathbf{x} \\
 \mathbf{x}^{e}\\
 \mathbf{u}
 \end{array}
 \left(
 \begin{array}{ccc}
 -I & K & 0 \\
 K_{I} & 0 & 0 \\
 I & 0 & 0 \\
 0 & I & 0 \\
 0 & 0 & I
 \end{array}
 \right)
 \end{equation}
 We use $\mathbf{w}^{e}$ and $\mathbf{x}^{e}$ to denote the auxiliary graph nodes corresponding to the states $\mathbf{x}^{e}(t)$.  We observe that in the augmented graph, since the weight matrix is diagonal, there is a directed edge from $w_{j}^{e,T}$ to $x_{j}^{e,T}$ for all edges indexed $j = 1,\ldots,M$.  We have the following intermediate result.
 \begin{lemma}
 \label{lemma:consensus_1}
 The matching $m$ with $m(w_{j}^{e,T}) = x_{j}^{e,T}$ for all $j=1,\ldots,M$ and $m(w_{i}^{T}) = w_{i}^{Q}$  for all $i=1,\ldots,N$ satisfies the conditions of Lemma \ref{lemma:bipartite_property}.
 \end{lemma}

 \begin{proof}
 The matching $m$ is valid under the construction of Section \ref{subsec:aux_graph}.  The rows of $\Omega$ from (\ref{eq:consensus_Omega}) indexed in $J = \{x_{j}^{e,Q}: j=1,\ldots,M\} \cup \{w_{i}^{Q} : i=1,\ldots,N\}$ form the matrix
 $$\Omega_{J} = \left(
 \begin{array}{cc}
 -I & K \\
 0 & I
 \end{array}
 \right),$$
 which has full rank.
 \end{proof}

 We can then compute $\Omega\Omega_{J \cup J_{1}}^{-1}$ as
 \begin{equation}
 \label{eq:consensus_Omega_J_J1}
 \Omega\Omega_{J \cup J_{1}}^{-1} = \left(
 \begin{array}{ccc}
 I & 0 & 0 \\
 -K_{I} & K_{I}K & 0 \\
 -I & K & 0 \\
 0 & I & 0 \\
 0 & 0 & I
 \end{array}
 \right)
 \end{equation}

 The graph $G^{\prime} = (V^{\prime}, E^{\prime})$ of Lemma \ref{lemma:equivalent_graph} defined by Eq. (\ref{eq:consensus_Omega_J_J1}) is described by the following lemma.
 \begin{lemma}
 \label{lemma:consensus_aux_graph}
 For the system (\ref{eq:consensus_LDS}), the edge set $E^{\prime}$ of Lemma \ref{lemma:equivalent_graph} is defined by
 \begin{eqnarray*}
 E^{\prime} &=& \{(w_{e}^{Q}, w_{i}^{Q}) : e = (i,j) \mbox{ for some $j \in V$}\} \cup \{(x_{i}^{Q}, x_{e}^{Q}) : e = (i,j) \mbox{for some $j \in V$}\} \\
 &&\cup \{(w_{e}^{Q}, x_{e^{\prime}}^{Q}) : \mbox{edges $e$ and $e^{\prime}$ have a node in common}\} \\
 && \cup \{(x_{j}^{e,Q}, x_{j}^{e,T}) : j=1,\ldots,M\} \cup \{(w_{j}^{e,T}, w_{j}^{e,Q}) : j=1,\ldots,M\} \\
 && \cup \{(w_{i}^{Q}, w_{i}^{T}) : i=1,\ldots,N\}.
 \end{eqnarray*}
 \end{lemma}

 \begin{proof}
 Considering the edge set of the auxiliary graph defined in Section \ref{subsec:aux_graph}, the set $\{(w_{i}^{T}, x_{j}^{Q}) : (i,j) \in N(T_{A}) \cup N(T_{F}), m(w_{i}^{T}) \neq x_{j}^{Q}\}$ is empty, while the set $\{(x_{j}^{Q}, w_{i}^{T}) : (i,j) \in N(T_{A}) \cup N(T_{F}), m(w_{i}^{T}) = x_{j}^{Q}\}$ is equal to $\{(x_{j}^{e,Q}, x_{j}^{e,T}) : j=1,\ldots,M\} \cup \{(w_{j}^{e,T}, w_{j}^{e,Q}) : j=1,\ldots,M\}$.  The set $\{(w_{i}^{T}, w_{i}^{Q}) : m(w_{i}^{T}) \neq w_{i}^{Q}\} \cup \{(w_{i}^{Q}, w_{i}^{T}) : m(w_{i}^{T}) = m(w_{i}^{Q})\}$ is equal to $\{(w_{j}^{e,T}, w_{j}^{e,Q}) : j=1,\ldots,M\} \cup \{(w_{i}^{Q}, w_{i}^{T}) : i=1,\ldots,N\}$.

 It remains to compute the set $\{(x^{Q},y^{Q}) : x \in \hat{V} \setminus J, y \in J, \tilde{\Omega}_{xy} \neq 0, \tilde{\Omega}_{xz} = 0 \ \forall z \in J_{1}\}$.  This set is defined by the off-diagonal entries of $\Omega\Omega_{J \cup J_{1}}^{-1}$ from (\ref{eq:consensus_Omega_J_J1}).  The entries from $w_{j}^{e,Q}$ to $w_{i}^{Q}$ correspond to the entries of the incidence matrix, and hence there is a nonzero entry if and only if edge $j$ is incident to node $i$.  A similar argument holds for the $(x_{i}^{Q}, x_{j}^{e,Q})$ edges.  Finally, an edge $(w_{j}^{e,Q}, w_{j^{\prime}}^{e,Q})$ is formed if $(K_{I}K)_{jj^{\prime}} \neq 0$.  This matrix is the \emph{edge Laplacian}, which has a nonzero entry if and only if either $j=j^{\prime}$, or edges $j$ and $j^{\prime}$ are incident to the same node.
 \end{proof}

 This description of the graph $G^{\prime}$ enables characterization of the input-connected nodes in $\hat{V}$.
 \begin{lemma}
 \label{lemma:consensus_input_connected}
  The nodes $x_{i}^{T}$, $w_{i}^{T}$ and $w_{i}^{Q}$  do not belong to any cycle in $G^{\prime}$. A node $x_{i}^{Q}$ is input-connected in $G^{\prime}$ if and only if $i$ is input-connected in the graph $G$ induced by the consensus dynamics. A node $x_{j}^{e,T}$, $w_{j}^{e,T}$, $x_{j}^{e,Q}$, or $w_{j}^{e,Q}$ is input-connected in $G^{\prime}$ if and only if edge $j$ is incident on a node that is input-connected in the consensus network $G$.
 \end{lemma}

 \begin{proof}
   By Lemma \ref{lemma:consensus_aux_graph}, $w_{i}^{T}$ and $w_{i}^{Q}$ are only connected to each other, and hence are not part of any cycle since the link is directional.  Similarly, the nodes $x_{i}^{T}$  cannot belong to any cycle, since they have no incoming edges.

   Now, suppose that there is a path from node $i$ to an input $i^{\prime}$ in $G$.  Let $(i, i_{1})$, $\ldots$, $(i_{r}, i^{\prime})$ denote one such path, and let $j_{0},\ldots,j_{r}$ denote the indices of the edges on the path.  Then there is a path $\pi$ from $x_{i}^{Q}$ to $w_{i^{\prime}}^{T}$, given by
   \begin{multline}
   \label{eq:consensus_aux_path}
    \pi = (x_{i}^{Q}, x_{j_{0}}^{e,Q}) \cup \bigcup_{l=0}^{r-1}{\left\{(x_{j_{l}}^{e,Q}, x_{j_{l}}^{e,T}), (x_{j_{l}}^{e,T}, w_{j_{l}}^{e,T}), (w_{j_{l}}^{e,T}, w_{j_{l}}^{e,Q}), (w_{j_{l}}^{e,Q}, x_{j_{l+1}}^{e,Q})\right\}} \\
    \cup \{(x_{j_{r}}^{e,Q}, w_{i^{\prime}}^{Q}), (w_{i^{\prime}}^{Q}, w_{i^{\prime}}^{T})\}
    \end{multline}

    For the other direction, we have that any path from $x_{i}^{Q}$ to $w_{i^{\prime}}^{T}$ has the form of (\ref{eq:consensus_aux_path}), and hence defines a path from $i$ to $i^{\prime}$ in $G$.  Finally, suppose that edge $j$ is incident on node $i$ and that there is a path from a node $i$ to an input $i^{\prime}$ in $G$.  Let $(i, i_{1}), \ldots, (i_{r}, i^{\prime})$ denote one such path, and let $j_{0},\ldots,j_{r}$ denote the indices of the edges on the path. For node $w_{j}^{e,Q}$, there is a path given by
    \begin{multline*}
    \pi^{\prime} = (w_{j}^{e,Q}, x_{j_{0}}^{e,Q}) \cup \bigcup_{l=0}^{r-1}{\{(x_{j_{l}}^{e,Q}, x_{j_{l}}^{e,T}), (x_{j_{l}}^{e,T}, w_{j_{l}}^{e,T}), (w_{j_l}^{e,T}, w_{j_l}^{e,Q}), (w_{j_{l}}^{e,Q}, x_{j_{l+1}}^{e,Q})\}} \\
    \cup \{(x_{j_{r}}^{e,Q}, w_{i^{\prime}}^{Q}), (w_{i^{\prime}}^{Q}, w_{i^{\prime}}^{T})\}
    \end{multline*}
    A path can also be found for nodes $x_{j}^{e,Q}$, $x_{j}^{e,T}$, and $w_{j}^{e,T}$ by using the path $(x_{j}^{e,Q}$, $x_{j}^{e,T})$, $(x_{j}^{e,T}, w_{j}^{e,T})$, $(w_{j}^{e,T}, w_{j}^{e,Q})$.
    \end{proof}

    Based on Lemma \ref{lemma:consensus_input_connected}, we can characterize exactly when the condition of Lemma \ref{lemma:connectivity_sufficient_general} holds, based on the connectivity of the network graph $G$.

    \begin{lemma}
    \label{lemma:consensus_connected}
    The condition of Lemma \ref{lemma:connectivity_sufficient_general} holds for the system (\ref{eq:consensus_LDS}) if and only if each node is input-connected in the graph $G$.
    \end{lemma}

The proof follows directly from Lemma \ref{lemma:consensus_input_connected}.  Lemma \ref{lemma:consensus_connected} enables us to improve the optimality bounds for a class of metrics with a certain structure.  Suppose that the connected components of the graph are equal to $G_{1},\ldots, G_{r}$, with $G_{i} = (V_{i}, E_{i})$ for $i=0,\ldots,r$.  We consider metrics of the form
\begin{equation}
\label{eq:special_metric}
f(S) = \sum_{i=1}^{r}{f_{i}(S \cap V_{i})}.
\end{equation}
Eq. (\ref{eq:special_metric}) has the interpretation that the performance of nodes in connected component $V_{i}$ only depends on the set of input nodes for component $V_{i}$, instead of the overall input set.  This structure holds for, e.g., the metrics of \cite{patterson2010leader,clark2014leader,clark2014convergence}.  For these metrics, we have the following optimality result.

\begin{theorem}
\label{theorem:consensus_optimality}
For the consensus system (\ref{eq:consensus_LDS}), the problem of maximizing a performance metric of the form (\ref{eq:special_metric}) subject to controllability as a constraint and $|S| \leq k$, formulated as
 \begin{equation}
 \label{eq:consensus_formulation}
 \begin{array}{ll}
 \mbox{maximize} & f(S) \\
 \mbox{s.t.} & (V \setminus S) \in \mathcal{M}_{1}^{\ast} \cap \mathcal{M}_{2}^{\ast} \\
  & |S| \leq k
  \end{array}
  \end{equation}
 can be approximated up to an optimality bound of $(1-1/e)$ in polynomial time.  As a special case, if the graph $G$ is strongly connected, then any monotone submodular performance metric can be approximated up to an optimality bound of $(1-1/e)$ in polynomial time.
\end{theorem}

\begin{proof}
The proof is by showing that the optimality bounds of Algorithm \ref{algo:Joint_PC} are improved in this case.  By Theorem \ref{theorem:Joint_PC}, the continuous relaxation phase of Algorithm \ref{algo:Joint_PC} returns a vector $y(1)$ such that $F(y(1)) \geq (1-1/e)f(S^{\ast})$, where $S^{\ast}$ is the optimal solution to (\ref{eq:consensus_formulation}).  Now, Theorem II.3 of \cite{chekuri2010dependent} implies that the SWAP\_ROUND subroutine satisfies $\mathbf{E}(f(S \cap Q)) \geq F(x_{i} : i \in Q)$ for any set $Q$ of equivalent elements of one of the matroids $\hat{\mathcal{M}}_{1}$ or $\hat{\mathcal{M}}_{2}$.  For the matroid $\hat{\mathcal{M}}_{2}$, each set of elements $V_{s}$ is equivalent, and so $\mathbf{E}(f(S \cap V_{s})) \geq F(x_{i} : i \in V_{s})$. Summing over $s$ yields the desired result.

For the special case, we have that when the graph is strongly connected, input connectivity holds provided there is at least one input. Hence we can  obtain a $(1-1/e)$-bound on the optimal input set.
\end{proof}

\subsection{Double Integrator Dynamics}
\label{subsec:double}
We study networked systems where the second derivative of each node's state $\xi_{i}(t)$ is a function of its neighbor states, so that $\ddot{\xi}_{i}(t) = \sum_{j \in N(i)}{W_{ij}\xi_{j}(t) + \Gamma_{ij}\dot{\xi}_{j}(t)}$, where $N(i)$ is the neighbor set of node $i$.  The double integrator model is applicable for larger vehicles that have inertial components in their dynamics \cite{ren2008double}.  We write this system in the form (\ref{eq:LDS}) by introducing variables $\zeta_{i}(t) = \dot{\xi}_{i}(t)$, resulting in dynamics
\begin{equation}
\label{eq:second_int_LDS}
\left(
\begin{array}{c}
\dot{\mathbf{\xi}}(t) \\
\dot{\mathbf{\zeta}}(t)
\end{array}
\right) =
\left(
\begin{array}{cc}
0 & I \\
W & \Gamma
\end{array}
\right)\left(
\begin{array}{c}
\mathbf{\xi}(t) \\
\mathbf{\zeta}(t)
\end{array}
\right).
\end{equation}
where $F=I$.
In analyzing this system, we observe that it is not possible to independently control the states $\xi_{i}(t)$ and $\zeta_{i}(t)$ to any arbitrary trajectories, since $\zeta_{i}(t) = \dot{\xi}_{i}(t)$.  Hence we assume that the state $\zeta_{i}(t)$ (the velocity) is controlled in the input nodes, while the state $\xi_{i}(t)$ continues to follow the dynamics (\ref{eq:second_int_LDS}).

We first investigate the auxiliary graph condition of Lemma \ref{lemma:connectivity_sufficient_general}.  The matrix $\Omega$ is given by
\begin{equation}
\label{eq:double_omega}
\Omega =
\begin{array}{c}
\mathbf{w}^{\xi} \\
\mathbf{w}^{\zeta} \\
\mathbf{x}^{\xi} \\
\mathbf{x}^{\zeta} \\
\mathbf{u}
\end{array}
\left(
\begin{array}{ccc}
-I & I & 0 \\
0 & -I & 0 \\
I & 0 & 0 \\
0 & I & 0 \\
0 & 0 & I
\end{array}
\right)
\end{equation}
We have that the rows indexed in $\mathbf{w}$ have full rank, and hence the matching $m(w_{i}^{\xi,T}) = w_{i}^{\xi,Q}$ and $m(w_{i}^{\zeta,T}) = w_{i}^{\zeta,Q}$ for $i=1,\ldots,N$ satisfies the conditions of Lemma \ref{lemma:bipartite_property}. This gives $J = \{w_{i}^{\zeta,Q}: i=1,\ldots,N\} \cup \{w_{i}^{\xi,Q} : i=1,\ldots,N\}$. We then have $\Omega_{J \cup J_{1}}$, $\Omega_{J \cup J_{1}}^{-1}$, and $\Omega\Omega_{J \cup J_{1}}^{-1}$ as
\begin{multline}
\label{eq:double_Omega}
\Omega_{J \cup J_{1}} = \left(
\begin{array}{ccc}
-I & I & 0 \\
0 & -I & 0 \\
0 & 0 & I
\end{array}
\right), \quad
\Omega_{J \cup J_{1}}^{-1} = \left(
\begin{array}{ccc}
-I & -I & 0 \\
0 & -I & 0 \\
0 & 0 & I
\end{array}
\right), \\
\Omega\Omega_{J \cup J_{1}}^{-1} =
\begin{array}{c}
\mathbf{w}_{\xi} \\
\mathbf{w}_{\zeta} \\
\mathbf{x}^{\xi} \\
\mathbf{x}^{\zeta} \\
\mathbf{u}
\end{array}
\left(
\begin{array}{ccc}
I & 0 & 0\\
0 & I & 0\\
-I & -I & 0 \\
0 & -I & 0 \\
0 & 0 & I
\end{array}
\right).
\end{multline}
These matrix values lead to the following description of the auxiliary graph.
\begin{lemma}
\label{lemma:double_auxiliary}
The graph $G^{\prime}$ of Lemma \ref{lemma:equivalent_graph} has edge set $E^{\prime}$ given by
\begin{eqnarray}
\label{eq:double_edge1}
E^{\prime} &=& \{(w_{i}^{\xi,Q}, w_{i}^{\xi,T}) : i=1,\ldots,N\} \cup \{(w_{i}^{\zeta,Q}, w_{i}^{\zeta,T}) : i=1,\ldots,N\} \\
\nonumber
  && \cup \{(x_{i}^{\xi,Q}, x_{i}^{\xi,T}) : i=1,\ldots,N\} \cup \{(x_{i}^{\zeta,Q}, x_{i}^{\zeta,T}) : i=1,\ldots,N\} \\
\label{eq:double_edge2}
&& \cup \{(w_{i}^{\xi,T}, x_{j}^{\xi,T}) : j \in N(i)\} \cup \{(w_{i}^{\xi, T}, x_{j}^{\zeta,T}) : j \in N(i)\} \\
\nonumber
&& \cup \{(w_{i}^{\zeta,T}, w_{j}^{\xi,T}) : j \in N(i)\} \cup \{(w_{i}^{\zeta,T}, w_{j}^{\xi,T}) : j \in N(i)\} \\
\label{eq:double_edge3}
&& \cup \{(x_{i}^{\xi,Q}, w_{i}^{\xi, Q}) : i=1,\ldots,N\} \cup \{(x_{i}^{\xi,Q}, w_{i}^{\zeta,Q}) : i=1,\ldots,N\}\\
\nonumber
 &&\cup \{(x_{i}^{\zeta,Q}, w_{i}^{\zeta,Q}) : i =1,\ldots,N\}
\end{eqnarray}
\end{lemma}

\begin{proof}
The edges (\ref{eq:double_edge1}) correspond to the edges $\{(w_{i}^{Q}, w_{i}^{T}) : m(w_{i}^{T}) = w_{i}^{Q}\}$ in the definition of $\hat{E}$. The edges enumerated in (\ref{eq:double_edge2}) correspond to the edges $\{(w_{i}^{T}, x_{j}^{T}) : (i,j) \in N(T_{A})$.  Finally, the value of $\Omega\Omega_{J \cup J_{1}}^{-1}$ from (\ref{eq:consensus_Omega_J_J1}) implies that the edges enumerated in (\ref{eq:double_edge3}) correspond to the edges $\{(x^{Q},y^{Q}) : x \in \hat{V} \setminus J, y \in J, \tilde{\Omega}_{xy} \neq 0, \tilde{\Omega}_{xz} = 0 \forall z \in J_{1}\}$.
\end{proof}

Lemma \ref{lemma:double_auxiliary} leads to the following result, which relates the connectivity of the auxiliary graph and the graph $G$ induced by the node dynamics.  This is analogous to Lemma \ref{lemma:consensus_input_connected}.

\begin{lemma}
\label{lemma:double_connected}
For any node $i$, the nodes $w_{i}^{\xi, Q}$, $w_{i}^{\zeta,Q}$, $x_{i}^{\xi, Q}$, $x_{i}^{\zeta,Q}$, $w_{i}^{\xi, T}$, $w_{i}^{\zeta,T}$, $x_{i}^{\xi, T}$, and $x_{i}^{\zeta,T}$ are input-connected in the auxiliary graph $G^{\prime}$ if and only if node $i$ is input-connected in the graph $G$.
\end{lemma}

\begin{proof}
Suppose that node $i$ is connected to an input node $i^{\prime}$ in $G$, with path $(i,i_{1}), (i_{1},i_{2}), \ldots, (i_{r},i^{\prime})$.  Now, consider the node $w_{i}^{\zeta,T}$.  We can construct a path $\pi$ from $w_{i}^{\zeta,T}$ to $w_{i^{\prime}}^{\xi,T}$ as
\begin{multline}
\label{eq:double_path}
\pi = (w_{i}^{\zeta,T}, x_{i_{0}}^{\xi,T}) \cup \bigcup_{l=0}^{r}{\{(x_{i_{l}}^{\xi,T}, x_{i_{l}}^{\xi,Q}), (x_{i_{l}}^{\xi,Q}, w_{i_{l}}^{\xi,Q}), (w_{i_{l}}^{\xi,Q}, w_{i_{l}}^{\xi,T}), (w_{i_{l}}^{\xi,T}, x_{i_{l+1}}^{\xi,T})\}} \\
\cup \{(x_{i^{\prime}}^{\xi,T}, x_{i^{\prime}}^{\xi,Q}), (x_{i^{\prime}}^{\xi,Q}, w_{i^{\prime}}^{\xi,Q}), (w_{i^{\prime}}^{\xi,Q}, w_{i^{\prime}}^{\xi,T})\}.
\end{multline}
Paths for the other types of nodes in the auxiliary graph can be constructed in a similar fashion. Conversely, any path to an input node in the auxiliary graph will have the form (\ref{eq:double_path}), and hence can be used to construct a path to an input node in the graph $G$.
\end{proof}


Lemma \ref{lemma:double_connected} implies that, for performance metrics satisfying (\ref{eq:special_metric}), Algorithm \ref{algo:Joint_PC} returns a set $S$ satisfying $f(S) \geq (1-1/e)f(S^{\ast})$, where $S^{\ast}$ is the optimal solution.  The proof is analogous to Theorem \ref{theorem:consensus_optimality}.

%

\subsection{Input Selection in Networks of Free Parameters}
\label{subsec:free_param}
We now investigate systems where all of the matrix entries are free parameters, as in the models of \cite{clark2012controllability,liu2011controllability,ruths2014control}.  We consider systems of the form $\dot{\mathbf{x}}(t) = A\mathbf{x}(t)$, where $A$ is a free matrix and $F=I$.  The matrix $\Omega$ defined in Section \ref{sec:preliminaries} is then given by
\begin{equation}
\label{eq:free_Omega}
\Omega =
\begin{array}{c}
\mathbf{w} \\
\mathbf{x}\\
\mathbf{u}
\end{array}
\left(
\begin{array}{cc}
-I & 0 \\
I & 0 \\
0 & I
\end{array}
\right)
\end{equation}
Hence the simple matching $m(w_{i}^{T}) = w_{i}^{Q}$ has full rank in (\ref{eq:free_Omega}), and we can compute $\Omega_{J \cup J_{1}}$ and $\Omega\Omega_{J \cup J_{1}}^{-1}$ as
\begin{equation}
\label{eq:free_Omega_J_J1}
\Omega_{J \cup J_{1}} = \left(
\begin{array}{cc}
-I & 0 \\
0 & I
\end{array}
\right), \quad \Omega\Omega_{J \cup J_{1}}^{-1} = \left(
\begin{array}{cc}
I & 0 \\
-I & 0 \\
0 & I
\end{array}
\right)
\end{equation}
Based on the value of $\Omega\Omega_{J \cup J_{1}}^{-1}$, the following lemma gives the construction of the auxiliary graph $G^{\prime}$.

\begin{lemma}
\label{lemma:free_auxiliary}
The auxiliary graph $G^{\prime}$ of Lemma \ref{lemma:connectivity_sufficient_general} has edge set $E^{\prime}$ given by
\begin{multline*}
E^{\prime} = \{(w_{i}^{Q}, w_{i}^{T}) : i=1,\ldots,N\} \cup \{(w_{i}^{T}, x_{j}^{T}) : j \in N(i)\} \cup \{(x_{i}^{T}, x_{i}^{Q}) : i=1,\ldots,N\}\\
 \cup \{(x_{i}^{Q}, w_{i}^{Q}) : i=1,\ldots,N\}
\end{multline*}
\end{lemma}

\begin{proof}
The first term follows from the matching $m(w_{i}^{T}) = w_{i}^{Q}$.  The second term arises from the matrix $A$, while the third term is from the definition of the auxiliary graph.  The last term follows from the value of $\Omega\Omega_{J \cup J_{1}}^{-1}$ in (\ref{eq:free_Omega_J_J1}).
\end{proof}

Hence, in the free parameter case we have a result analogous to Lemmas \ref{lemma:consensus_input_connected} and \ref{lemma:double_connected}.

\begin{lemma}
\label{lemma:free_connected}
The conditions of Lemma \ref{lemma:connectivity_sufficient_general} are met if and only if each node is path-connected to an input node in the graph $G$.
\end{lemma}

\begin{proof}
Suppose that a node $i$ is path-connected to an input node $i^{\prime}$ in $G$, where the path is given by $(i, i_{0}),\ldots,(i_{r-1},i_{r}), (i_{r},i^{\prime})$, letting $i_{r+1} = i^{\prime}$.  Then the corresponding path in $G^{\prime}$ is equal to
\begin{eqnarray*}
\pi &=& \{(w_{i}^{Q}, w_{i}^{T}), (w_{i}^{T}, x_{i_{0}}^{T}), (x_{i_{0}}^{T}, x_{i_{0}}^{Q}), (x_{i_{0}}^{Q}, w_{i_{0}}^{Q})\} \\
&& \cup \bigcup_{l=0}^{r}{\{(w_{i_{l}}^{Q}, w_{i_{l}}^{T}), (w_{i_{l}}^{T}, x_{i_{l+1}}^{T}), (x_{i_{l+1}}^{T}, x_{i_{l+1}}^{Q}), (x_{i_{l+1}}^{Q}, w_{i_{l+1}}^{Q})\}} \\
&& \cup (w_{i^{\prime}}^{Q}, w_{i^{\prime}}^{T})
\end{eqnarray*}
\end{proof}

We now investigate the other matroid constraint of Lemma \ref{lemma:matroid_non_input}.  When all entries of $A$ are free, the constraint reduces to $$\mbox{rank}(\mathcal{M}([I|0|0]) \vee \mathcal{M}([I|T_{A}|T_{B}(S)])) = 2n,$$ which is in turn equivalent to the second matroid being full rank.  Hence the condition can be reduced to $\mbox{rank}(\mathcal{M}([I|T_{A}|T_{B}(S)])) = n$.  

%
%

Lemma \ref{lemma:free_connected} implies that, if a system with free parameters is strongly connected, then it suffices to find an input set such that the matroid of Lemma \ref{lemma:matroid_non_input} is full rank.  A minimum-size input set such that $\mbox{rank}(\mathcal{M}([I|T_{A}|T_{B}(S)])) = n$ can be found efficiently using the greedy algorithm. This enables efficient computation of an input set with the same size as in \cite{liu2011controllability}, but through the matroid optimization framework.

 Finally, we observe that, by Lemma \ref{lemma:free_connected}, Algorithm \ref{algo:Joint_PC} returns a set $S$ that satisfies a $(1-1/e)$ optimality bound, by the argument of Theorem \ref{theorem:consensus_optimality}. 

%
\section{Numerical Study}
\label{sec:numerical}
We numerically evaluated our framework using Matlab.  We considered the consensus network case of Section \ref{subsec:consensus}.  Network topologies were generated by placing nodes at uniform random positions within a square region, and creating a link $(i,j)$ if node $i$ is within the communication range of node $j$.  The range of each node was chosen uniformly at random from the interval $[0,600]$.  We investigated the minimum-size set of input nodes for structural controllability, as well as selection of input nodes for joint performance and controllability.

\begin{figure}[!ht]
\centering
\includegraphics[width=3in]{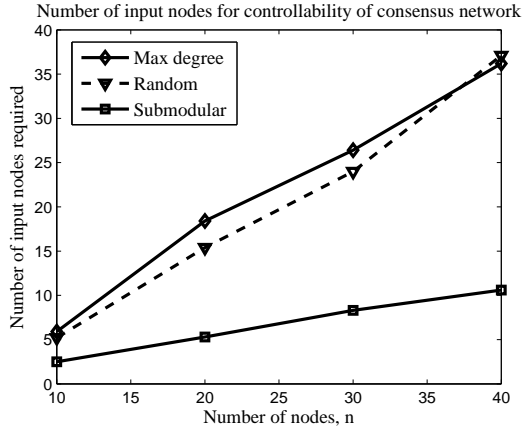}
\caption{Minimum-size input set for structural controllability in a consensus network. The submodular optimization approach is compared to max degree-based and random input selection.  The submodular optimization approach typically requires roughly one-quarter of the network to be controlled, while the random and degree-based heuristics select nearly all network nodes before controllability is satisfied.}  
\label{fig:controll}
\end{figure}

In the case of selecting the minimum-size set of input nodes for structural controllability, we considered networks of size $n = \{10,20,30,40\}$.  The deployment area was selected to yield average node degrees $d=3$.   We compared our submodular optimization approach with selecting high degree nodes as inputs, as well as selecting random nodes to act as inputs.  The submodular optimization approach required fewer input nodes to satisfy controllability, with the other heuristics selecting nearly all network nodes before controllability is satisfied.  For all schemes, the number of input nodes was increasing in the network size. 

\begin{figure}[!ht]
\centering
\includegraphics[width=3in]{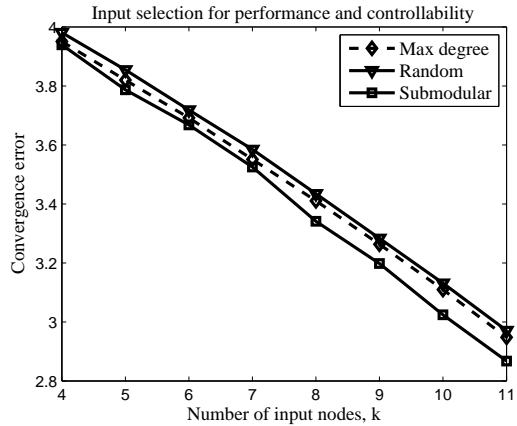}
\caption{Convergence error when input nodes are selected to minimize convergence error while satisfying controllability. The number of nodes $n$ was equal to $20$.  The submodular optimization approach provided lower convergence error than degree-based and random selection algorithms, especially as the number of input nodes increased.} 
\label{fig:joint_PC}
\end{figure}

We evaluated selection of input nodes in order to minimize the convergence error with controllability as a constraint.  The convergence error was defined as $||\mathbf{x}(t) - x^{\ast}\mathbf{1}||_{2}$, where $x^{\ast}$ is the state of the input nodes, $t=1$, and the initial state and edge weights were chosen uniformly at random.  The number of nodes was equal to $20$, while the deployment area was chosen to achieve an average degree of $2$. The submodular optimization approach provided lower convergence error than the degree-based and random heuristics, while also satisfying controllability from the input set.  As the number of input nodes increased, the gap between the submodular optimization approach and the other heuristics increased.   
\section{Conclusions}
\label{sec:conclusion}
In this paper, we studied the problem of input selection for joint performance and controllability of structured linear descriptor systems.  Our main contribution was to prove that structural controllability of linear descriptor systems can be mapped to two matroid constraints, representing controllability of the zero and nonzero modes of the system.  We demonstrated that, by exploiting this matroid structure, a minimum-size set of input nodes to guarantee structural controllability of such systems can be selected in polynomial time via matroid intersection algorithms.  We further showed that selection of input nodes for joint performance and controllability can be formulated as a submodular maximization problem subject to two matroid basis constraints.  We presented polynomial-time algorithms for obtaining a continuous solution to the input selection problem, providing a $(1-1/e)$ optimality bound, which can then be rounded to obtain a feasible input set.  We demonstrated that, when the objective function is modular, the optimal input selection for performance and controllability can be computed in polynomial time.

We investigated input selection in systems where the graph representation of the system is strongly connected, and found that for almost all systems of this type the number of matroid constraints can be reduced from two to one.  This  led to an $O(n)$ algorithm for selecting a minimum-size set of input nodes for structural controllability, as well as more efficient polynomial-time algorithms for approximating the optimal input set up to a factor of $(1-1/e)$.  We studied linear consensus systems, double integrator systems, and systems consisting of free parameters within our framework, and showed that the additional structure of each system provided a provable $(1-1/e)$ optimality bound for input selection based on performance and controllability.

\bibliographystyle{siam}
\bibliography{SIAM_Controllability}

\section{Appendix}
In this appendix, we provide proofs of Lemma \ref{lemma:connectivity_sufficient_general} and Theorem \ref{theorem:Joint_PC}.

 The steps in the proof of Lemma \ref{lemma:connectivity_sufficient_general} follow those in \cite{murota1987refined}.  First, define a set of coefficients on the edges $\hat{E}$ of graph $\hat{G}$.  The coefficient $\gamma(e)$ on edge $e \in \hat{E}$ is defined by
 \begin{equation}
 \label{eq:gamma}
 \gamma(e) = \left\{
 \begin{array}{ll}
 r_{i}, & e = (w_{i}^{T},w_{i}^{Q}), m(w_{i}^{T}) \neq w_{i}^{Q} \\
 -r_{i}, & e = (w_{i}^{T}, w_{i}^{Q}), m(w_{i}^{T}) = w_{i}^{Q} \\
 c_{i}, & e = (x_{i}^{T}, x_{i}^{Q}), i \notin J \\
 -c_{i}, & e = (x_{i}^{T}, x_{i}^{Q}), i \in J \\
 1, & e = (w_{i}^{T}, x_{j}^{T}), m(w_{i}^{T}) \neq x_{j}^{Q} \\
 -1, & e = (w_{i}^{T}, x_{j}^{T}), m(w_{i}^{T}) = x_{j}^{Q}\\
 0, & \mbox{else}
 \end{array}
 \right.
 \end{equation}

 Now, let $\tilde{\Omega} = \Omega\Omega_{J \cup J_{1}}^{-1}$ and $V^{-} = \{v^{Q} : \tilde{\Omega}_{vj} \neq 0 \mbox{ for some } j \in J_{1}\}$.    The following appears as Theorem 4.7 of \cite{murota1987refined}, albeit with slightly modified notation.

 \begin{theorem}
 \label{theorem:murota}
 Let $\tilde{V}$ denote the set of vertices that are not reachable to $V^{-}$, and let $\tilde{G}$ denote the subgraph of $\hat{G}$ induced by $\tilde{V}$.  Then the condition $\mbox{rank}((zF-A)|B) = n$ holds for almost any choice of the free parameters if and only if the sum of the $\gamma(e)$ along any directed cycle in $\tilde{G}$ is zero.
 \end{theorem}

 We now prove Lemma \ref{lemma:connectivity_sufficient_general}.

 \begin{proof}[Proof of Lemma \ref{lemma:connectivity_sufficient_general}] If the condition of Lemma \ref{lemma:connectivity_sufficient_general} holds, then $\tilde{V}$ does not contain any cycle.  Hence the condition of Theorem \ref{theorem:murota} holds automatically, and we have $\mbox{rank}((zF-A)|B) = n$ for almost any free parameters.
 \end{proof}

In order to prove Theorem \ref{theorem:Joint_PC}, we first prove a sequence of lemmas, which will establish that $F(\mathbf{y}(1)) \geq (1-1/e)f(S^{\ast})$.  The lemmas follow Lemmas 3.1--3.3 of \cite{calinescu2011maximizing}, however, the results of \cite{calinescu2011maximizing} are for a single matroid constraint, instead of two matroid basis constraints as in Theorem \ref{theorem:Joint_PC}.

\begin{lemma}
\label{lemma:Joint_PC_1}
Let $y \in [0,1]^{n}$ and let $R \subseteq V$ denote a random set such that $j \in R$ with probability $y_{j}$.  Then $$f(S^{\ast}) \leq F(y) + \max_{I \in \hat{\mathcal{B}}_{1} \cap \hat{\mathcal{B}}_{2}}{\sum_{j \in I}{\mathbf{E}(f_{R}(j))}},$$ where $f_{R}(j) = f(R \cup \{j\}) - f(R)$.
\end{lemma}

\begin{proof}
By submodularity, we have that $f(S^{\ast}) \leq f(R) + \sum_{j \in S^{\ast}}{f_{R}(j)}$.  Taking expectation over $R$ yields
$$f(S^{\ast}) \leq \mathbf{E}(f(R)) + \sum_{j \in S^{\ast}}{\mathbf{E}(f_{R}(j))} \leq F(y) + \max_{I \in \mathcal{I}}{\sum_{j \in I}{\mathbf{E}(f_{R}(j))}},$$ as desired.
\end{proof}

Lemma \ref{lemma:Joint_PC_1} establishes an optimality result for an idealized version of Algorithm \ref{algo:Joint_PC} where computation is performed over the actual values of $\mathbf{E}(f_{R}(j))$ instead of estimates computed via random sampling.  The following lemma introduces bounds on the errors introduced by sampling.

\begin{lemma}
\label{lemma:Joint_PC_2}
With high probability, at each time $t$ the algorithm finds a set $I(t)$ such that $$\sum_{j \in I(t)}{\mathbf{E}(f_{R(t)}(j))} \geq (1-2k\delta)f(S^{\ast}) - F(y(t)).$$
\end{lemma}

The proof is identical to Lemma 3.2 of \cite{calinescu2011maximizing} and is omitted.  Finally, we prove the optimality bound on the solution to the continuous relaxation.

\begin{lemma}
\label{lemma:Joint_PC_3}
With high probability, the fractional solution $y(1)$ found by solving the continuous problem satisfies $$F(y) \geq \left(1- \frac{1}{e} - \frac{1}{3d}\right)f(S^{\ast}).$$
\end{lemma}

The proof follows that of Lemma 3.3 of \cite{calinescu2011maximizing}.  We now prove Theorem \ref{theorem:Joint_PC}.

\begin{proof}[Proof of Theorem \ref{theorem:Joint_PC}]
The fact that $F(y) \geq \left(1 - \frac{1}{3}\right)f(S^{\ast})$ follows directly from Lemma \ref{lemma:Joint_PC_3}.  The feasibility of $S$ follows from the SWAP\_ROUND algorithm, which preserves membership in the bases of both matroids at each iteration and returns an integral solution, corresponding to a common basis of $\hat{\mathcal{B}}_{1}$ and $\hat{\mathcal{B}}_{2}$.  The complexity of the procedure is dominated by the solution of the continuous problem, which is in turn determined by the cost of computing the maximum weight matroid intersection at each iteration.  Since there are $O(n^{2})$ iterations and the maximum weight matroid intersection has complexity $O(\tau n^{3})$ \cite{schrijver2003combinatorial}, where $\tau$ is the cost of testing independence in $\hat{\mathcal{M}}_{1}$ and $\hat{\mathcal{M}}_{2}$, the overall complexity is $O(\tau n^{5})$.
\end{proof} 

\end{document}